\documentclass[11pt, reqno]{amsart}
\usepackage{amssymb, amsthm, amsmath, amsfonts,mathtools}
\usepackage{array, epsfig}
\usepackage{bbm}
\usepackage{enumitem}

\usepackage{hyperref}
\usepackage[numbers,square]{natbib}
\usepackage{ stmaryrd }
\allowdisplaybreaks

\numberwithin{equation}{section}

  \evensidemargin
\oddsidemargin
\newcommand{\stirling}[2]{\genfrac{[}{]}{0pt}{}{#1}{#2}}
\newcommand{\stirlingsec}[2]{\genfrac{\{}{\}}{0pt}{}{#1}{#2}}

\newcommand{\stirlingb}[2]{B\hspace*{-0.2mm}\big[{#1},{#2}\big]}
\newcommand{\stirlingsecb}[2]{B\hspace*{-0.2mm}\big\{{#1},{#2}\big\}}

\newcommand{\N}{\mathbb{N}}

\newcommand{\R}{\mathbb{R}}

\newcommand{\1}{\mathbbm{1}}

\newcommand{\eps}{\varepsilon}

\newcommand*\xbar[1]{%
   \hbox{%
     \vbox{%
       \hrule height 0.5pt 
       \kern0.25ex
       \hbox{%
         \kern-0.05em
         \ensuremath{#1}%
         \kern-0.1em
       }%
     }%
   }%
}

\DeclareMathOperator{\supp}{supp}

\DeclareMathOperator{\sgn}{sgn}
\DeclareMathOperator{\lin}{lin}
\DeclareMathOperator{\conv}{conv}
\DeclareMathOperator{\pos}{pos}
\DeclareMathOperator{\inte}{int}
\DeclareMathOperator{\relint}{relint}

\DeclareMathOperator{\rank}{rank}

\newcommand{\cA}{\mathcal{A}}
\newcommand{\cB}{\mathcal{B}}
\newcommand{\cC}{\mathcal{C}}

\newcommand{\cF}{\mathcal{F}}

\newcommand{\cL}{\mathcal{L}}

\newcommand{\cN}{\mathcal{N}}

\newcommand{\cP}{\mathcal{P}}
\newcommand{\cR}{\mathcal{R}}

\newcommand{\cT}{\mathcal{T}}

\renewcommand{\P}{\mathbb{P}}

\newcommand{\aff}{\mathop{\mathrm{aff}}\nolimits}

\newcommand{\ind}{\mathbbm{1}}

\theoremstyle{plain}
\newtheorem{theorem}{Theorem}[section]
\newtheorem{lemma}[theorem]{Lemma}
\newtheorem{corollary}[theorem]{Corollary}
\newtheorem{proposition}[theorem]{Proposition}

\theoremstyle{definition}

\theoremstyle{remark}

\makeindex
\makeglossary

\begin{document}

\author{Thomas Godland}
\address{Thomas Godland: Institut f\"ur Mathematische Stochastik,
Westf\"alische Wilhelms-Universit\"at M\"unster,
Orl\'eans-Ring 10,
48149 M\"unster, Germany}
\email{thomas.godland@uni-muenster.de}

\author{Zakhar Kabluchko}
\address{Zakhar Kabluchko: Institut f\"ur Mathematische Stochastik,
Westf\"alische {Wilhelms-Uni\-ver\-sit\"at} M\"unster,
Orl\'eans--Ring 10,
48149 M\"unster, Germany}
\email{zakhar.kabluchko@uni-muenster.de}

\title[Projections and angle sums]{Projections and angle sums of belt polytopes and permutohedra}




\begin{abstract}
Let $P\subset \R^n$ be a belt polytope, that is a polytope whose normal fan coincides with the fan of some hyperplane arrangement $\mathcal A$. Also, let $G:\R^n\to\R^d$ be a linear map of full rank whose kernel is in general position with respect to the faces of $P$.  We derive a formula for the number of $j$-faces of the ``projected'' polytope $GP$ in terms of the $j$-th level characteristic polynomial of $\mathcal A$. In particular, we  show that the face numbers of $GP$ do not depend on the linear map $G$ provided a general position assumption is satisfied. Furthermore, we derive formulas for the sum of the conic intrinsic volumes and Grassmann angles of the tangent cones of $P$ at all of its $j$-faces.  We apply these results to permutohedra of types $A$ and $B$, which yields closed formulas for the face numbers of projected permutohedra and the generalized angle sums of permutohedra in terms of Stirling numbers of both kinds and their $B$-analogues.
\end{abstract}

\maketitle

\section{Introduction}


In the work of Donoho and Tanner~\cite{Donoho2009} the following interesting statement can be found: The number of $j$-faces of the image of the $n$-dimensional cube $[0,1]^n$ under a linear map $G:\R^n\to\R^d$ does not depend on the choice of $G$ provided $G$ is in ``general position''. That is to say,  the cube is an equiprojective polytope as defined by Shephard~\cite{Shephard1967}. More precisely, by~\cite[Eq.~(1.6)]{Donoho2009} we have
\begin{equation}\label{eq:donoho_tanner}
f_j(G [0,1]^n)=2\binom{n}{j}\sum_{\ell=n-d}^{n-j-1}\binom{n-j-1}{\ell};
\end{equation}
for all  $0\le j < d\le n$, where $f_j(P)$ denotes the number of $j$-faces of a polytope $P$, and $G [0,1]^n$ is the image of $[0,1]^n$ under $G$.


We will show that this result can be generalized to any polytope $P\subset \R^n$ whose normal fan, that is the set of normal cones at all faces  of $P$, coincides with the fan of some linear hyperplane arrangement $\cA$. These polytopes are called \emph{belt polytopes}; see~\cite[page~226]{Ziegler_LecturesOnPolytopes}, \cite{bolker}, \cite{baladze_solution_engl}, \cite[Chapter~VII]{boltyanski_book}, \cite{backman}. It is known~\cite[Theorem~7.16]{Ziegler_LecturesOnPolytopes} that zonotopes, that is Minkowski sums of finitely many line segments, are special cases of belt polytopes. One simple example is the cube $[0,1]^n$ appearing in the formula~\eqref{eq:donoho_tanner}.  Another examples of belt polytopes are  permutohedra of types $A$ and $B$ which are defined by
\begin{align*}
\cP_n^A:=\cP_n^A(x_1,\dots,x_n):=\conv\{(x_{\sigma(1)},\dots,x_{\sigma(n)}):\sigma\in\text{Sym}(n)\big\}
\end{align*}
and
\begin{align*}
\cP_n^B:=\cP_n^B(x_1,\dots,x_n):=\conv\{(\eps_1x_{\sigma(1)},\dots,\eps_nx_{\sigma(n)}):(\eps_1,\dots,\eps_n)\in\{\pm 1\}^n,\sigma\in\text{Sym}(n)\big\}
\end{align*}
for  a point $(x_1,\dots,x_n)\in\R^n$. Here, $\text{Sym}(n)$ denotes the group of all permutations of the set $\{1,\dots,n\}$.
The faces and normal fans of the permutohedra can be  characterized in terms of reflection arrangements of types $A_{n-1}$ and $B_n$, respectively. This will allow us to obtain analogues of~\eqref{eq:donoho_tanner} for permutohedra.

The paper is organized as follows.
Section~\ref{section:preliminaries} introduces the necessary notation and some well-known definitions and results from convex and integral geometry. In Section~\ref{section:main_results}, we state the main results in the setting of belt polytopes.
In Section~\ref{section:permutohedra}, we specialize these results to permutohedra. In the same section,  we recall various characterizations of permutohedra, their faces and normal fans. Also, we show that permutohedra are zonotopes only in some rare exceptional cases, namely if the numbers $x_1,\ldots,x_n$ form an arithmetic sequence. Proofs are given in Sections~\ref{sec:proofs:belt_poly} and~\ref{sec:proofs_permutoh}.


\section{Preliminaries}\label{section:preliminaries}

In this section, we are going to introduce necessary facts and notation from convex and integral geometry. These facts are well-known and can be skipped at first reading.

\subsection{Facts from convex geometry}
For a set $M\subset \R^n$ denote by $\lin M$ and $\aff M$ the linear hull and the affine hull of $M$, respectively. They are defined as the minimal linear and the minimal affine subspace of $\R^n$ containing $M$, respectively. Equivalently, $\lin M$ can be defined as the set of all linear combinations of elements in $M$, while $\aff M$ can be defined as the set of all affine combinations of elements in $M$. Similarly, the convex hull of $M$ is denoted by $\conv M$ and defined as the minimal convex set containing $M$, or, equivalently,
\begin{align*}
\conv M:=\big\{\lambda_1x_1+\ldots+\lambda_mx_m:m\in\N,x_1,\dots,x_m\in M,\lambda_1+\ldots+\lambda_m\ge 0,\lambda_1+\ldots +\lambda_m=1\big\}.
\end{align*}
The positive hull of a set $M$ is denoted by $\pos M$ and defined as
\begin{align*}
\pos M:=\big\{\lambda_1x_1+\ldots+\lambda_mx_m:m\in\N,x_1,\dots,x_m\in M,\lambda_1,\ldots,\lambda_m\ge 0\big\}.
\end{align*}
The relative interior of a set $M$ is the set of all interior points of $M$ relative to its affine hull $\aff M$ and it is denoted by $\relint M$. The set of interior points of $M$ is denoted by $\inte M$.

A \emph{polyhedral set} is an intersection of finitely many closed half-spaces (whose boundaries need not pass through the origin). A bounded polyhedral set is called \emph{polytope}. Equivalently, a polytope can be defined as the convex hull of a finite set of points. A \emph{polyhedral cone} (or just cone) is an intersection of finitely many closed half-spaces whose boundaries contain the origin and therefore a special case of the polyhedral sets. Equivalently, a polyhedral cone can be defined as the positive hull of a finite set of points. The dimension of a polyhedral set $P$ is defined as the dimension of its affine hull $\aff P$.

A supporting hyperplane for a polyhedral set $P\subset \R^n$ is an affine hyperplane $H$ with the property that $H\cap P\neq\varnothing$ and $P$ lies entirely in one of the closed half-spaces bounded by $H$. A \emph{face} of a polyhedral set $P$ (of arbitrary dimension) is a set of the form $F=P\cap H$, for a supporting hyperplane $H$, or the set $P$ itself. Equivalently, the faces of a polyhedral set $P$ are obtained by replacing some of the half-spaces, whose intersection defines the polyhedral set, by their boundaries and taking the intersection. The set of all  faces of $P$ is denoted by $\cF(C)$ and the set of all $k$-dimensional faces (or just $k$-faces) of $P$ by $\cF_k(P)$ for $k\in\{0,\dots,n\}$. The number of $k$-faces of $P$ is denoted by $f_k(C):=\#\cF_k(C)$. In general, the number of elements in a set $M$ is denoted by $|M|$ or $\#M$.  The $0$-dimensional faces are called \emph{vertices}. In the case of a cone, the only possible vertex is the origin.

The \emph{dual cone $C^\circ$} (or polar cone) of a cone $C\subset\R^n$ is defined as
\begin{align*}
C^\circ:=\{v\in\R^n:\langle v,x\rangle\le 0\;\forall x\in C\}.
\end{align*}
The \emph{tangent cone $T_F(P)$} of a polyhedral set $P\subset\R^n$ at a face $F$ of $P$ is defined by
\begin{align*}
T_F(P)=\{x\in\R^n:f_0+\eps x\in P\text{ for some }\eps>0\},
\end{align*}
where $f_0$ is any point in the relative interior of $F$. This definition does not depend on the choice of $f_0$. The \emph{normal cone} of $P$ at the face $F$ is defined as the dual of the tangent cone, that is
$
N_F(P):=T_F(P)^\circ.
$



\subsection{Grassmann angles and conic intrinsic volumes}\label{section:grassmann}

Now, we are going to introduce some important geometric functionals of cones. Let $C\subset \R^n$ be a cone and $g$ be an $n$-dimensional standard Gaussian  vector. Then, the \emph{$k$-th conic intrinsic volume} of $C$ is defined as
\begin{align*}
\upsilon_k(C):=\sum_{F\in\cF_k(C)}\P(\Pi_C(g)\in\relint F),\quad k=0,\dots,n,
\end{align*}
where $\Pi_C$ denotes the orthogonal projection on $C$, that is, $\Pi_C(x)$ is the vector in $C$ which minimizes the Euclidean distance to $x\in\R^n$.

The conic intrinsic volumes are the analogues of the usual intrinsic volumes in the setting of conical or spherical geometry. Equivalently, the conic intrinsic volumes can be defined using the spherical Steiner formula, as done in~\cite[Section 6.5]{SW08}.
For further properties of conic intrinsic volumes we refer to~\cite[Section~2.2]{Amelunxen2017} and~\cite[Section~6.5]{SW08}.

Following Grünbaum~\cite{gruenbaum_grass_angles}, we define the \emph{Grassmann angles} $\gamma_k(C)$, $k\in\{0,\dots,n\}$, of a cone $C$ as follows. Let $W_{n-k}$ be random linear subspace of $\R^n$ with uniform distribution on the Grassmannian of all $(n-k)$-dimensional subspaces. Then, the $k$-th Grassmann angle of $C$ is defined as
\begin{align*}
\gamma_k(C):=\P(W_{n-k}\cap C\neq \{0\}),\quad k=0,\dots,n.
\end{align*}
If the lineality space $C\cap -C$ of a cone $C$, which is the maximal linear subspace contained in $C$, has dimension $j\in\{0,\dots,n-1\}$, the Grassmann angles satisfy
\begin{align*}
1=\gamma_0(C)=\ldots=\gamma_j(C)\ge\gamma_{j+1}(C)\ge\ldots\ge \gamma_n(C)=0.
\end{align*}
As proved in~\cite[Eq.~(2.5)]{gruenbaum_grass_angles}, the Grassmann angles do not depend on the dimension of the ambient linear subspace. This means that if we embed $C$ in $\R^N$ with $N\ge n$, we obtain the same Grassmann angles. Therefore, it is convenient to write $\gamma_N(C):=0$ for all $N\ge \dim C$. If $C$ is not a linear subspace, then $\frac{1}{2}\gamma_k(C)$ is also known as the \emph{$k$-th conic quermassintegral}; see~\cite[Eqs.~(1)-(4)]{SW08} or~\cite{HugSchneider2016}.

The conic intrinsic volumes and Grassmann angles satisfy a linear relation, called the \emph{conic Crofton formula}.
More precisely, we have
\begin{align}\label{eq:conic_crofton}
\gamma_k(C)=2\sum_{i=1,3,5,\dots}\upsilon_{k+i}(C)
\end{align}
for all $k\in\{0,\dots,n\}$ and for every cone $C\subset\R^n$ which is not a linear subspace, according to~\cite[p.261]{SW08}. Consequently,
\begin{align}\label{eq:conic_crofton_corollary}
\upsilon_k(C)=\frac{1}{2}\gamma_{k-1}(C)-\frac 12\gamma_{k+1}(C),
\end{align}
for all $k\in\{0,\dots,n\}$, where in the cases $k=0$ and $k=n$ we have to define $\gamma_{-1}(C)=1$ and $\gamma_{n+1}(C) = 0$. Then,~\eqref{eq:conic_crofton_corollary} follows from~\eqref{eq:conic_crofton} and the identity $\upsilon_0(C)+\upsilon_2(C)+\ldots=1/2$; see~\cite[Eq.~(5.3)]{amelunxen_edge}.



\section{Main results on belt polytopes}\label{section:main_results}


\subsection{Belt polytopes and hyperplane arrangements}
Given a polytope $P\subset \R^n$ and a linear map $G:\R^n\to \R^d$ we are interested in determining the face numbers of the ``projected'' polytope $GP$. It will be shown below that if $P$ is a belt polytope, then the face numbers of $GP$ are independent of the linear map $G$ provided it satisfies some minor general position assumption. Let us first define the class of belt polytopes and some related notions which will be needed below.

A (linear) \emph{hyperplane arrangement} is a finite collection of distinct linear hyperplanes in $\R^n$. The \emph{lattice} $\cL(\cA)$ generated by $\cA$ consists of all linear subspaces of $\R^n$ that can be represented as intersections of hyperplanes from $\cA$ including $\R^n$ (which is an intersection over the empty set).
Denote  by $\cL_j(\cA)$ the set of $j$-dimensional linear subspaces from the lattice $\cL(\cA)$, for $j\in \{0,\ldots, n\}$.  By convention, we put $\cL_{n}(\cA):=\{\R^n\}$.

The hyperplanes from $\cA$ dissect $\R^n$ into finitely many cones or chambers. More precisely, the complement $\R^n\backslash\bigcup_{H\in\cA}H$ is a disjoint union of open convex sets. The set of closures of these ``regions'' is denoted by $\cR(\cA)$ and called the \emph{conical mosaic} generated by $\cA$. The elements of $\cR(\cA)$ are called \emph{chambers}.

In general, a \emph{fan} in $\R^n$ is defined as a family $\cF$ of non-empty cones with the following two properties:
\begin{itemize}
\item[(i)] Every non-empty face of a cone in $\cF$ is also a cone in $\cF$;
\item[(ii)] The intersection of any two cones in $\cF$ is a face of both;
\end{itemize}
see~\cite[Chapter~7]{Ziegler_LecturesOnPolytopes}.
For example, the \emph{fan generated by a hyperplane arrangement} $\cA$ is defined as the collection of all faces of its chambers, that is
\begin{align*}
\cF(\cA)=\bigcup_{C\in\cR(\cA)} \cF(C).
\end{align*}
We denote the set of $j$-dimensional cones from $\cF(\cA)$ by $\cF_j(\cA)$. On the other hand,  the \emph{normal fan} of a non-empty polytope $P\subset\R^n$ is defined as the set of normal cones of all faces of $P$, that is
\begin{align*}
\cN(P):=\{N_F(P):F\in \cF(P)\}.
\end{align*}

A polytope $P\subset \R^n$ is called a \emph{belt polytope} if its normal fan $\cN(P)$ coincides with the fan $\cF(\cA)$ generated by some hyperplane arrangement $\cA$. Originally, belt polytopes were defined as polytopes such that every $2$-face contains together with each edge another parallel edge; see~\cite[page~226]{Ziegler_LecturesOnPolytopes}, \cite{backman} for the equivalence of the definitions. Examples of belt polytopes are zonotopes (see~\cite[Theorem~7.16]{Ziegler_LecturesOnPolytopes}) and permutohedra of types $A$ and $B$ (see Section~\ref{section:permutohedra}).


\subsection{General position assumptions}\label{section:gen_pos}
To state our results on face numbers of projected belt polytopes, we need to introduce the terminology of \emph{general position} in the context of hyperplane arrangements and polyhedral sets.
Let $M$ be an affine subspace of $\R^n$. Denote by $L\subset\R^n$ the unique linear subspace such that
$M=t+L$
for some $t\in\R^n$, that is the translation of $M$ passing through the origin.
We say that $M$ is in \emph{general position with respect to a linear subspace $L'\subset\R^n$} if
\begin{align}\label{eq:def_gen_pos_to_arrangement}
\dim(L\cap L')=\max(\dim L+\dim L'-n,0).
\end{align}
A linear subspace $L'$ is said to be in \emph{general position with respect to a polyhedral set $P$} if the affine hull of each face  $F$ of $P$ is in general position with respect to $L'$.
On the other hand, a linear subspace $L'\subset\R^n$ is said to be in \emph{general position with respect to the hyperplane arrangement $\cA$}  if $L'$ is in general position with respect to each $L\in\cL(\cA)$.


Now, we are ready state two equivalent general position assumptions that we need to impose on a linear mapping $G: \R^{n} \to\R^d$ under consideration (actually, both assumptions deal with the kernel of $G$ only). 

\begin{theorem}\label{theorem:gen_pos_general}
Let $P\subset\R^n$ be a belt polytope and denote its normal fan by $\cN(P)$. Let $\cA$ be the linear hyperplane arrangement whose fan $\cF(\cA)$ coincides with $\cN(P)$. For $1\le d\le \dim P$ and a linear map $G: \R^{n} \to \R^d$ with $\rank G=d$ the following two general position assumptions are equivalent:
\begin{enumerate}[label=(G\arabic*), leftmargin=50pt]
\item The $(n-d)$-dimensional linear subspace $\ker G$ is in general position with respect to $P$. \label{label_GPG1}
\item The $d$-dimensional linear subspace $(\ker G)^\perp$ is in general position with respect to the hyperplane arrangement $\cA$.\label{label_GPG2}
\end{enumerate}
\end{theorem}

The proof of this theorem is postponed to Section~\ref{section:proof_general_position}.

\subsection{Face numbers of projected belt polytopes}
Before stating the result, we need to introduce the characteristic polynomial of a hyperplane arrangement.
The \emph{rank} of a linear hyperplane arrangement $\cA$ in $\R^n$ is defined by
\begin{align*}
\rank(\cA)=n-\dim\bigg(\bigcap_{H\in\cA}H\bigg),\quad\rank(\varnothing)=0.
\end{align*}
The \emph{characteristic polynomial} $\chi_\cA(t)$ of $\cA$ can be defined by the following Whitney formula:
\begin{align}\label{eq:def_char_polynomial_whitney}
\chi_\cA(t)=\sum_{\cC\subset \cA}(-1)^{\#\cC}t^{n-\rank(\cC)};
\end{align}
see, e.g., \cite[Lemma~2.3.8]{Orlik1992} or~\cite[Theorem~2.4]{Stanley2007},  as well as~\cite[Section~1.3]{Stanley2007} or~\cite[Section~3.11.2]{Stanley2015} for other definitions using the M\"obius function on the intersection poset of $\cA$. More generally, following Amelunxen and Lotz~\cite{Amelunxen2017}, the \emph{$m$-th level characteristic polynomial} of $\cA$ is defined by
\begin{equation}
\chi_{\cA,m}(t) := \sum_{M\in\cL_{m}(\cA)} \chi_{A|M}(t),
\end{equation}
where $m\in \{0,\ldots, n\}$ and $\chi_{A|M}(t)$ is the characteristic polynomial of the induced hyperplane arrangement $\cA|M:=\{H\cap M:H\in\cA,M\nsubseteq H\}$ in the ambient space $M$. We use the following notation for the coefficients of $\chi_{\cA,m}(t)$:
\begin{equation}\label{eq:char_poly_hyperplane_arr_m_level_coefficients}
\chi_{\cA,m}(t)=\sum_{k=0}^{m}(-1)^{m-k} a_{k,m} t^k.
\end{equation}
We define $a_{k,m} := 0$ for $k\notin\{0,\ldots, m\}$.

\begin{theorem}\label{theorem:face_numbers_general_polytope}
Let $P\subset\R^n$ be a belt polytope and denote its normal fan by $\cN(P)$. Let $\cA$ be the linear hyperplane arrangement whose fan $\cF(\cA)$ coincides with $\cN(P)$. Moreover, let $G:\R^{n} \to \R^d$ be a linear map with $\rank G=d$ such that the equivalent general position assumptions~\ref{label_GPG1} and~\ref{label_GPG2} are satisfied. Then, for $0\le j<d\le \dim P$, the number of $j$-faces of the projected polytope $GP$ is independent of the linear map $G$ and given by
\begin{equation}\label{eq:face_number_proj_belt_polytope}
f_j(GP)
=
2(a_{n-d+1, n-j}+a_{n-d+3, n-j}+\ldots).
\end{equation}
\end{theorem}

The series in~\eqref{eq:face_number_proj_belt_polytope} terminates after finitely many non-zero terms. The proof of Theorem~\ref{theorem:face_numbers_general_polytope} is postponed to Section~\ref{section:proof_face_numbers}. 
As a consequence of Theorem~\ref{theorem:face_numbers_general_polytope}, the belt polytopes are equiprojective as defined in~\cite{Shephard1967}.

The coefficients of the $m$-th level characteristic polynomial appear in the following lemma which will be used in the proof of Theorem~\ref{theorem:face_numbers_general_polytope}.
\begin{lemma}\label{lemma:number_faces_arr_intersected_by_lin_subspace}
Let $L_d\subset\R^n$ be a linear subspace of dimension $d\in \{0,\ldots, n\}$ that is in general position with respect to a hyperplane arrangement $\cA$. Then, for all $m\in \{0,\ldots, n\}$,  the number of $m$-faces of  the fan $\cF(\cA)$ that are non-trivially intersected by $L_d$ is given by
\begin{align}
\#\{F\in\cF_m(\cA): \relint F\cap L_d\neq\varnothing\}
&=
2(a_{n-d+1,m}+a_{n-d+3,m}+\ldots) \notag\\
&=
\#\{F\in\cF_m(\cA): F\cap L_d \neq\{0\}\}, \label{eq:lemma:faces_arr_intersected_by_lin_subspace}
\end{align}
where the $a_{k,m}$'s are defined by~\eqref{eq:char_poly_hyperplane_arr_m_level_coefficients}. If we drop the general position assumption, then both equalities should be replaced by $\leq$.
\end{lemma}

\subsection{Angle sums of belt polytopes}
In the next theorem we compute the generalized angle sums of belt polytopes.
Recall that the Grassmann angles $\gamma_d$ and the conic intrinsic volumes $\upsilon_d$ were defined in Section~\ref{section:grassmann}.
\begin{theorem}\label{theorem:angle_sums_belt_polytope}
Let $P\subset\R^n$ be a belt polytope with normal fan $\cN(P)$ and let $\cA$ be the linear hyperplane arrangement whose fan $\cF(\cA)$ coincides with $\cN(P)$. Then, for all $0\le j\leq d\le \dim P$ we have
\begin{align*}
\sum_{F\in\cF_j(P)}\upsilon_d(T_F(P))
&=
a_{n-d, n-j},\\
\sum_{F\in\cF_j(P)}\gamma_d(T_F(P))
&=
2(a_{n-d-1, n-j}+a_{n-d-3,n-j}+\ldots).
\end{align*}
\end{theorem}
The proof is postponed to Section~\ref{sec:proof_angle_sums_belt_polys}.  Let us mention that applying Theorem~\ref{theorem:angle_sums_belt_polytope} to a full-dimensional zonotope $P$ with $d=n$ we recover a formula stated in~\cite[Theorem~12]{KS11}, while the special case $d=j$ is covered by~\cite[Proposition~3.2]{backman}.

\section{Applications to permutohedra of types \texorpdfstring{$A$}{A} and \texorpdfstring{$B$}{B}}\label{section:permutohedra}

In this section, we apply the results of Section~\ref{section:main_results} to permutohedra of types $A$ and $B$. These polytopes have been studied starting with the work of Schoute~\cite{Schoute1911} in 1911; see~\cite{billera,Postnikov2009,Hohlweg2011} as well as~\cite[Example~0.10]{Ziegler_LecturesOnPolytopes}, \cite[Section~5.3]{yemelichev_book}, \cite[pp.~58--60, 254--258]{barvinok_book} and~\cite[Example~2.2.5]{Bjoerner1999}.

\subsection{Definitions of permutohedra}
Take some point $(x_1,\dots,x_n)\in\R^n$. A \emph{permutohedron of type $A$} is defined as the following polytope in $\R^n$:
\begin{align*}
\cP_n^A=\cP_n^A(x_1,\dots,x_n):=\conv\big\{(x_{\sigma(1)},\dots,x_{\sigma(n)}):\sigma\in\text{Sym}(n)\big\},
\end{align*}
where $\text{Sym}(n)$ is the group of all permutations of the set $\{1,\dots,n\}$. Note that $\cP_n^A$ is contained in the hyperplane $\{t\in\R^n:t_1+\ldots+t_n=x_1+\ldots+x_n\}$ and therefore has dimension at most $n-1$. Similarly, a \emph{permutohedron of type $B$} is defined as the following polytope in $\R^n$:
\begin{align*}
\cP_n^B=\cP_n^B(x_1,\dots,x_n):=\conv\big\{(\eps_1x_{\sigma(1)},\dots,\eps_nx_{\sigma(n)}):\eps=(\eps_1,\dots,\eps_n)\in\{\pm1 \}^n,\sigma\in\text{Sym}(n)\big\}.
\end{align*}
Note that $\cP_n^A(x_1,\dots,x_n)$ does not change under permutations of $x_1,\ldots,x_n$, whereas $\cP_n^B(x_1,\dots,x_n)$ stays invariant under signed permutations. Therefore, there is no restriction of generality if we  assume that $x_1\geq \ldots\geq x_n$ in the $A$-case and $x_1\geq \ldots \geq x_n\geq 0$ in the $B$-case.

The next lemma is due to Rado~\cite{Rado1952} (see also~\cite[Section~5.3]{yemelichev_book}, \cite[p.~257]{barvinok_book} and~\cite[Corollary B.3]{marshall_olkin_arnold_book}) and describes $\cP_n^A$ as a set of solutions to a finite system of affine inequalities.
\begin{lemma}\label{lemma:rado_A}
Assume that $x_1\ge\ldots\ge x_n$. Then, a point $(t_1,\dots,t_n)\in\R^n$ belongs to the permutohedron $\cP_n^A(x_1,\dots,x_n)$ of type $A$ if and only if
\begin{align*}
t_1+\dots+t_n=x_1+\ldots+x_n
\end{align*}
and, for every non-empty subset $M\subset \{1,\dots,n\}$, we have
\begin{align*}
\sum_{i\in M}t_i\le x_1+\ldots+x_{|M|}.
\end{align*}
\end{lemma}

An analogous result for the permutohedron of type $B$, together with a proof and references to the original literature, can be found in~\cite[Corollary~C.5.a]{marshall_olkin_arnold_book}.

\begin{lemma}\label{lemma:rado_B}
Assume that $x_1\ge\ldots\ge x_n\ge 0$. Then, a point $(t_1,\dots,t_n)\in\R^n$ belongs to the permutohedron $\cP_n^B(x_1,\dots,x_n)$ of type $B$ if and only if for every non-empty subset $M\subset\{1,\dots,n\}$, we have
\begin{align}\label{eq:rado_B}
\sum_{i\in M}|t_i|\le x_1+\ldots+x_{|M|}.
\end{align}
\end{lemma}

\subsection{Normal fans of permutohedra}

The following theorems characterize the normal fans of the permutohedra of types $A$ and $B$ and show that both types of permutohedra are belt polytopes.


\begin{theorem}\label{theorem:normalfan_permuto_A}
For $x_1>\ldots>x_n$ the normal fan $\cN(\cP_n^A(x_1,\dots,x_n))$ of the permutohedron of type $A$ coincides with the fan $\cF(\cA(A_{n-1}))$ generated by the hyperplane arrangement $\cA(A_{n-1})$ consisting of the hyperplanes
\begin{align}\label{eq:def_refl_arr_A}
\{\beta\in\R^n:\beta_i=\beta_j\},\quad 1\le i<j\le n.
\end{align}
\end{theorem}

\begin{theorem}\label{theorem:normalfan_permuto_B}
For $x_1>\ldots>x_n>0$ the normal fan $\cN(\cP_n^B(x_1,\dots,x_n))$ of the permutohedron of type $B$ coincides with the fan $\cF(\cA(B_n))$ generated by the hyperplane arrangement $\cA(B_n)$  consisting  of the hyperplanes
\begin{align}\label{eq:def_refl_arr_B}
&\{\beta\in\R^n:\beta_i=\beta_j\},\quad 1\le i<j\le n,\notag\\
&\{\beta\in\R^n:\beta_i=-\beta_j\},\quad 1\le i<j\le n,\\
&\{\beta\in\R^n:\beta_i=0\},\quad 1\le i\le n.\notag
\end{align}
\end{theorem}
The arrangements $\cA(A_{n-1})$ and $\cA(B_n)$ are called \emph{reflection arrangements} of types $A_{n-1}$ and $B_n$ and  the cones they generate are called the Weyl chambers.
Both theorems seem to be known, see, e.g., \cite[Section~3.1]{Hohlweg2011}, but for the sake of completeness we will give their proofs in Section~\ref{subsec:proof_normal_fans_permutoh}. For example, the normal cones at the vertices of the permutohedra coincide with the Weyl chambers of types $A_{n-1}$ and $B_n$, which was used to compute their statistical dimension in~\cite[Proposition~3.5]{amelunxen_edge}.

\subsection{General position assumptions}
In the next two corollaries of Theorem~\ref{theorem:gen_pos_general} we state general position assumptions we need to impose on the linear map $G:\R^n\to\R^d$ when computing the face numbers of projected permutohedra $G\cP_n^A$ and $G\cP_n^B$.

\begin{corollary}\label{cor:gen_pos_A}
Let $1\leq d\le n-1$ and $x_1>\ldots>x_n$. For a linear map $G:\R^{n} \to \R^d$ with $\rank G=d$, the following two conditions are equivalent:
\begin{enumerate}[label=(A\arabic*), leftmargin=50pt]
\item The $(n-d)$-dimensional linear subspace $\ker G$ is in general position with respect to $\cP_n^A(x_1,\dots,x_n)$. \label{label_GPA1}
\item The $d$-dimensional linear subspace $(\ker G)^\perp$ is in general position with respect to the reflection arrangement $\mathcal{A}(A_{n-1})$ defined in~\eqref{eq:def_refl_arr_A}.\label{label_GPA2}
\end{enumerate}
\end{corollary}


\begin{corollary}\label{cor:gen_pos_B}
Let $1\le d\le n$ and $x_1>\ldots>x_n>0$. For a linear map $G:\R^{n} \to \R^d$ with $\rank G=d$, the following two conditions are equivalent:
\begin{enumerate}[label=(B\arabic*), leftmargin=50pt]
\item The $(n-d)$-dimensional linear subspace $\ker G$ is in general position with respect to $\cP_n^B(x_1,\dots,x_n)$. \label{label_GPB1}
\item The $d$-dimensional linear subspace $(\ker G)^\perp$ is in general position with respect to the reflection arrangement $\mathcal{A}(B_n)$ defined in~\eqref{eq:def_refl_arr_B}.\label{label_GPB2}
\end{enumerate}
\end{corollary}

Both corollaries follow immediately from Theorem~\ref{theorem:gen_pos_general} since the normal fans $\cN(\cP_n^A)$ and $\cN(\cP_n^B)$ coincide with $\cF(\cA(A_{n-1}))$ and $\cF(\cA(B_n))$, respectively. Note that if $x_1>\ldots>x_n$, then the permutohedron $\cP_n^A$ has dimension $n-1$ (since any point $t$ in the hyperplane  $t_1+ \ldots + t_n = x_1+\ldots + x_n$ which is sufficiently close to the point $(\bar x_n,\ldots, \bar x_n)$, where $\bar x_n = (x_1+\ldots +x_n)/n$, satisfies the inequalities from Lemma~\ref{lemma:rado_A}). Similarly, for $x_1>\ldots>x_n>0$, the permutohedron $\cP_n^B$ of type $B$ has dimension $n$.




\subsection{Face numbers of projected permutohedra}\label{section:face_numbers}
In this section, we state our results on the number of faces of projected permutohedra of types $A$ and $B$.
The formulas will be stated in terms of Stirling numbers defined as follows.



The (signless) \emph{Stirling number of the first kind} $\stirling{n}{k}$ is the number of permutations of the set $\{1,\dots,n\}$ having exactly $k$ cycles. Equivalently, these numbers can be defined as the coefficients of the polynomial
\begin{align}\label{eq:def_stirling1_polynomial}
t(t+1)\cdot\ldots \cdot (t+n-1)=\sum_{k=0}^n\stirling{n}{k}t^k
\end{align}
for $n\in\N_0$, with the convention that $\stirling{n}{k}=0$ for $n\in \N_0$, $k\notin\{0,\dots,n\}$ and $\stirling{0}{0} = 1$. The \emph{Stirling number of the second kind} $\stirlingsec{n}{k}$  is defined as the number of partitions of the set $\{1,\dots,n\}$ into $k$ non-empty subsets.

The \emph{$B$-analogues of the Stirling numbers}, denoted by $\stirlingb nk$ and $\stirlingsecb nk$, are defined by the formulas
\begin{align}\label{eq:def_stirling12b}
(t+1)(t+3)\cdot\ldots\cdot (t+2n-1)=\sum_{k=0}^n \stirlingb nk t^k,
\quad
\stirlingsecb nk=\sum_{m=k}^{n}2^{m-k}\binom{n}{m}\stirlingsec{m}{k}.
\end{align}
for $n\in\N$ and, by convention, $\stirlingb nk=\stirlingsecb nk = 0$ for $k\notin\{0,\dots,n\}$.
The triangular arrays $\stirlingb nk$ and $\stirlingsecb nk$ appear as entries A028338 and  A039755 in~\cite{sloane}; see also~\cite{suter}.

\begin{lemma}
For all $j\in \{0,\ldots, n\}$, the number of $j$-dimensional subspaces in the lattices $\cL(\cA(A_{n-1}))$ and $\cL(\cA(B_{n}))$ is given by $\stirlingsec{n}{j}$ and $\stirlingsecb{n}{j}$, respectively.
\end{lemma}
\begin{proof}
The linear subspaces from $\cL_j(\cA(A_{n-1}))$ are in one-to-one correspondence with the (unordered) partitions $\cB = \{B_1,\ldots, B_j\}$ of $\{1,\ldots, n\}$ into non-empty, disjoint subsets. The linear subspace corresponding to $\cB$ is given by
$$
\{x\in \R^n: x_{i_1} = x_{i_2} \text{ for all } i_1, i_2 \text{ contained in the same block } B_\ell\}.
$$
The number of such partitions is $\stirlingsec{n}{j}$. Similarly, the elements of $\cL_j(\cA(B_{n}))$ are in one-to-one correspondence with the equivalence classes of pairs $(\cB, \eta)$, where $\cB=\{B_1,\ldots, B_j\}$ is an (unordered) partition of some subset $\supp \cB:= B_1\cup \ldots \cup B_j$ of $\{1,\ldots, n\}$ into disjoint, non-empty subsets $B_1,\ldots, B_j$, and $\eta:\supp \cB\to \{\pm 1\}$ is a sign function. Moreover,  two pairs $(\cB, \eta')$ and $(\cB, \eta'')$ with the same $\cB$ are considered equivalent if $\eta'/\eta''$ stays constant on each block $B_\ell$. The $j$-dimensional linear subspace corresponding to $(\cB, \eta)$ is given by
$$
\{x\in \R^n: \eta(i_1) x_{i_1} = \eta(i_2) x_{i_2} \text{ for all } i_1, i_2 \text{ contained in the same } B_\ell;
\;
x_i=0 \text{ for all } i\notin \supp \cB\}.
$$
Denoting the cardinality of $\supp \cB$ by $m\in \{j, \ldots, n\}$ we see that the number of equivalence classes as above is given by $\stirlingsecb nj$ as defined in~\eqref{eq:def_stirling12b}. Note that $2^{m-j}$ is the number of equivalence classes of sign functions $\eta$.
\end{proof}

Let $m\in \{0,\ldots, n\}$. The $m$-th level characteristic polynomials of the reflection arrangements $\cA(A_{n-1})$ and $\cA(B_n)$ are known from the work of Amelunxen and Lotz~\cite[Lemma~6.5]{Amelunxen2017}:
\begin{align}
\chi_{\cA(A_{n-1}),m}(t) &= \stirlingsec{n}{m} t(t-1)\ldots (t-m+1) = \stirlingsec{n}{m} \sum_{k=0}^m (-1)^{m-k} \stirling{m}{k} t^k,\label{eq:char_poly_level_reflection_arr_A}\\
\chi_{\cA(B_{n}),m}(t) &= \stirlingsecb{n}{m}  (t-1)(t-3)\ldots (t-2m+1) = \stirlingsecb{n}{m} \sum_{k=0}^m (-1)^{m-k} \stirlingb{m}{k} t^k.
\label{eq:char_poly_level_reflection_arr_B}
\end{align}
Note that in~\cite{Amelunxen2017}, the cardinality of $\cL_j(\cA(B_n))$ has been calculated incorrectly  (the power of $2$ in~\eqref{eq:def_stirling12b} is missing there).
It follows from~\eqref{eq:char_poly_level_reflection_arr_A} and~\eqref{eq:char_poly_level_reflection_arr_B} that the coefficients $a_{k,m}$, as defined in~\eqref{eq:char_poly_hyperplane_arr_m_level_coefficients}, are given by the following formulas:
\begin{align}
&\text{for } \cA(A_{n-1}):  &  a_{k,m} &= \stirlingsec{n}{m} \stirling{m}{k},\qquad k\in \{0,\ldots, m\}, \label{eq:coeff_char_poly_level_reflection_arr_A}\\
&\text{for } \cA(B_{n}):    &  a_{k,m} &= \stirlingsecb{n}{m} \stirlingb{m}{k}, \qquad k\in \{0,\ldots, m\}.
\label{eq:coeff_char_poly_level_reflection_arr_B}
\end{align}
Applying Theorem~\ref{theorem:face_numbers_general_polytope} to permutohedra and using~\eqref{eq:coeff_char_poly_level_reflection_arr_A} and~\eqref{eq:coeff_char_poly_level_reflection_arr_B}, we obtain the following results on the face numbers of projected permutohedra.
\begin{theorem}\label{theorem:facenumbers_proj_permuto_A}
Let $x_1>\ldots>x_n$ be given. For a linear map $G:\R^{n} \to \R^d$ with $\rank G=d$ which satisfies one of the equivalent general position assumptions~\ref{label_GPA1} or~\ref{label_GPA2}, we have
\begin{equation}\label{eq:face_num_proj_permutoh_type_A}
f_j(G\cP_n^A)=2\stirlingsec{n}{n-j}\left(\stirling{n-j}{n-d+1}+\stirling{n-j}{n-d+3}+\ldots\right),
\end{equation}
for all $0\le j<d\le n-1$.
\end{theorem}

\begin{theorem}\label{theorem:facenumbers_proj_permuto_B}
Let $x_1>\ldots>x_n>0$ be given. For a linear map $G:\R^{n} \to \R^d$ with $\rank G=d$ which satisfies one of the equivalent general position assumptions~\ref{label_GPB1} or~\ref{label_GPB2}, we have
\begin{equation}\label{eq:face_num_proj_permutoh_type_B}
f_j(G\cP_n^B)=2\stirlingsecb{n}{n-j}\big(\stirlingb{n-j}{n-d+1}+\stirlingb{n-j}{n-d+3}+\ldots\big),
\end{equation}
for all $0\le j<d\le n$.
\end{theorem}



\subsection{Faces of reflection arrangements intersected by linear subspaces}
The next lemma is closely related to Theorems~\ref{theorem:facenumbers_proj_permuto_A} and~\ref{theorem:facenumbers_proj_permuto_B} (as we shall see from the proofs) and follows from Lemma~\ref{lemma:number_faces_arr_intersected_by_lin_subspace}  in view of~\eqref{eq:coeff_char_poly_level_reflection_arr_A} and~\eqref{eq:coeff_char_poly_level_reflection_arr_B}.
\begin{lemma}\label{lemma:Weyl_faces_intersected_by_L}
The number of $j$-faces of Weyl chambers  (where each face is counted exactly once) intersected non-trivially by a $d$-dimensional subspace $L_d$ in general position with respect to the reflection arrangement $\cA(A_{n-1})$, respectively, $\cA(B_n)$, is given by
\begin{align*}
&\text{for } \cA(A_{n-1}): \sum_{F\in\cF_j(\cA(A_{n-1}))}\ind_{\{F\cap L_d\neq\{0\}\}}=2\stirlingsec{n}{j}\bigg(\stirling{j}{n-d+1}+\stirling{j}{n-d+3}+\ldots\bigg),\\
&\text{for } \cA(B_{n}): \sum_{F\in\cF_j(\cA(B_n))}\ind_{\{F\cap L_d\neq\{0\}\}}=2\stirlingsecb nj\big(\stirlingb{j}{n-d+1}+\stirlingb{j}{n-d+3}+\ldots\big),
\end{align*}
for all $j,d\in \{1,\ldots,n\}$.
\end{lemma}
The special case $j=n$ of this lemma can be found in~\cite[Theorem~3.4]{KVZ15} or~\cite[Theorem 2.4]{KVZ2019_multi_analogue}. A related result in a setting where the faces are counted with certain non-trivial multiplicities can be found in~\cite[Theorems~2.1, 2.8]{KVZ2019_multi_analogue}.

\subsection{Angle sums of permutohedra}\label{section:angle_sums}
In the next two theorems we compute the generalized angle sums of permutohedra. Both results follow from Theorem~\ref{theorem:angle_sums_belt_polytope} in view of~\eqref{eq:coeff_char_poly_level_reflection_arr_A} and~\eqref{eq:coeff_char_poly_level_reflection_arr_B}.


\begin{theorem}\label{theorem:angle_sums_permutoh_A}
Let $x_1>\ldots>x_n$ be given. Then, for all $0\le j\leq  d\le n-1$,
\begin{align}
\sum_{F\in\cF_j(\cP_n^A)}\upsilon_d(T_F(\cP_n^A))
&=
\stirlingsec{n}{n-j}\stirling{n-j}{n-d}, \label{eq:intr_vol_sums_permutoh_A}\\
\sum_{F\in\cF_j(\cP_n^A)}\gamma_d(T_F(\cP_n^A))
&=
2\stirlingsec{n}{n-j}\sum_{l=0}^\infty\stirling{n-j}{n-d-2l-1}. \label{eq:grassmann_sums_permutoh_A}
\end{align}
\end{theorem}

\begin{theorem}\label{theorem:angle_sums_permutoh_B}
Let $x_1>\ldots>x_n>0$ be given. Then, for all $0\le j\leq  d\le n$,
\begin{align}
\sum_{F\in\cF_j(\cP_n^B)}\upsilon_d(T_F(\cP_n^B))
&=
\stirlingsecb{n}{n-j}\stirlingb{n-j}{n-d},\label{eq:intr_vol_sums_permutoh_B}\\
\sum_{F\in\cF_j(\cP_n^B)}\gamma_d(T_F(\cP_n^B))
&=
2\stirlingsecb{n}{n-j}\sum_{l=0}^\infty \stirlingb{n-j}{n-d-2l-1}.\label{eq:grassmann_sums_permutoh_B}
\end{align}
\end{theorem}

\subsection{Permutohedra and zonotopes}
Besides permutohedra, the zonotopes are also a special case of the class of belt polytopes. A \emph{zonotope} $Z=Z(V)\subset\R^n$ is a Minkowski sum of a finite number of line segments, and therefore, can be written as
\begin{align*}
Z(V)=[-v_1,v_1]+\ldots+[-v_p,v_p]+z
\end{align*}
for some $p\in \N$, a matrix $V=(v_1,\dots,v_p)\in\R^{n\times p}$ and $z\in\R^n$. By~\cite[Definition~7.13]{Ziegler_LecturesOnPolytopes}, a zonotope $Z=Z(V)$ can equivalently be defined as the image of a cube under an affine map, that is,
\begin{align*}
Z(V):=V[-1,+1]^p+z=\{Vy+z:y\in[-1,+1]^p\}.
\end{align*}
In the book of Ziegler~\cite[Theorem~7.16]{Ziegler_LecturesOnPolytopes} it is proved that for a zonotope $Z=Z(V)\subset\R^n$, the normal fan $\cN(Z)$ of $Z$ coincides with the fan $\cF(\cA)$ of the hyperplane arrangement
\begin{align*}
\cA=\cA_V:=\{H_1,\dots,H_p\}
\end{align*}
in $\R^n$, where $H_i:=\{x\in\R^n:\langle x,v_i\rangle = 0\}$ for $i=1,\dots,p$. It is known~\cite[Example~7.15]{Ziegler_LecturesOnPolytopes}
that $\cP_n^{A}(n,n-1,\ldots,2, 1)$ is a zonotope and the natural question arises if permutohedra of types $A$ and $B$ are zonotopes for all $(x_1,\dots,x_n)\in\R^n$. The following theorem shows that this is not the case. Its proof is postponed to Section~\ref{subsec:proofs_permotohedra_as_zonotopes}.

\begin{theorem}\label{theo:permuto_vs_zonotopes}
For $x_1>\ldots>x_n$, the permutohedron $\cP_n^A(x_1,\dots,x_n)$ of type $A$ is a zonotope if and only if $x_1,\dots,x_n$ are in arithmetic progression, that is,
\begin{align}\label{eq:arithm_progr}
x_1=a+(n-1)b,\,x_2=a+(n-2)b,\dots,x_{n-1}=a+b,\,x_n=a
\end{align}
for some $a\in\R$ and $b>0$.

For $x_1>\ldots>x_n>0$, the permutohedron $\cP_n^B(x_1,\dots,x_n)$ of type $B$ is a zonotope if and only if $x_1,\dots,x_n$ are in arithmetic progression, that is if~\eqref{eq:arithm_progr} holds
for some $a>0$ and $b>0$.
\end{theorem}

\section{Proofs of the results on belt polytopes}\label{sec:proofs:belt_poly}
\subsection{General position: Proof of Theorem~\ref{theorem:gen_pos_general}}\label{section:proof_general_position}
Let $F\in\cF_k(P)$ be a $k$-face of $P$ for some $k\in\{0,\dots,\dim P\}$ and let $L$ be the linear subspace parallel to $\aff F$ with the same dimension as $\aff F$, that is, $\aff F=t+L$ for some $t\in\R^n$. Then, the normal cone $N_F(P)$ is $(n-k)$-dimensional and coincides with some $(n-k)$-dimensional cone from the fan of $\cA$, that is, an $(n-k)$-face of the conical mosaic generated by $\cA$.
Thus, $\lin N_F(P)$ can be represented as an intersection of hyperplanes from $\cA$ and therefore is an element of the lattice $\cL(\cA)$.

On the other hand, by definition, $T_F(P)$ contains the linear subspace $L$. Thus, we have $(\aff F)^\perp=L^\perp\supset T_F(P)^\circ=N_F(P)$. Since both $N_F(P)$ and $(\aff F)^\perp$ are $(n-k)$-dimensional, we obtain
\begin{align*}
L^\perp=(\aff F)^\perp=\lin N_F(P)\in\cL_{n-k}(\cA).
\end{align*}
The same argumentation applied backwards shows that, conversely,  each $(n-k)$-dimensional subspace $K\in\cL(\cA)$ coincides with  $\lin N_F(P)$ for some $k$-face $F$ of $P$. If we write $\aff F=t+L$ for some $t\in\R^n$, as above, then  we obtain $K=(\aff F)^\perp=L^\perp$.

The equivalence of (G1) and (G2) follows easily from these observations. Condition~\ref{label_GPG1} is not satisfied if and only if
\begin{align*}
\dim(L\cap\ker G)\neq\max\{k-d,0\}
\end{align*}
for some $k\in\{0,\dots,\dim P\}$ and some $k$-dimensional linear subspace $L$ such that $\aff F=t+L$ for some $F \in \cF_k(P)$ and $t\in\R^n$. By the above observation, $L^\perp\in\cL_{n-k}(\cA)$ and
\begin{align*}
\dim\big((\ker G)^\perp\cap L^\perp\big)
&	=n-\dim(L+\ker G )\\
&	=n-\left(\dim(\ker G)+\dim L-\dim(L\cap\ker G)\right)\\
&	=d-k+\dim(L\cap \ker G)\\
&   \neq d-k +  \max\{k-d,0\}\\
&= \max\{0,d-k\}.
\end{align*}
Thus, $(\ker G)^\perp$ is not in general position with respect to $\cA$ and~\ref{label_GPG2} is not satisfied. Since every $K\in\cL(\cA)$ can be represented as $L^\bot$ as above, the same argument applies backwards.
\hfill$\Box$


\subsection{Proof of Lemma~\ref{lemma:number_faces_arr_intersected_by_lin_subspace}}
We need the following  formula for the number of \textit{full-dimensional} chambers of a hyperplane arrangement that are intersected by a linear subspace non-trivially. 
For its proof, we refer to~\cite[Theorem~3.1]{greene_zaslavsky} or~\cite[Theorem~3.3]{KVZ15} in combination with~\cite[Lemma~3.5]{KVZ15}.

\begin{lemma}\label{lemma:number_regions}
Let $L_d\subset\R^n$ be a linear subspace of dimension $d\in \{0,\ldots, n\}$ that is in general position with respect to a hyperplane arrangement $\cA$. Recall that  $\cR(\cA)$ is the set of closed chambers generated by $\cA$.  Then,
\begin{align*}
\#\{R\in\cR(\cA):\inte R\cap L_d\neq\varnothing\}
=
2(a_{n-d+1}+a_{n-d+3}+\ldots)
=
\#\{R\in\cR(\cA): R\cap L_d\neq\{0\}\}
,
\end{align*}
where the $a_k$'s are the coefficients of the characteristic polynomial $\chi_\cA(t)=\sum_{k=0}^n(-1)^{n-k}a_kt^k$. If we drop the general position assumption, then both equalities should be replaced by $\leq$.
\end{lemma}

\begin{proof}[Proof of Lemma~\ref{lemma:number_faces_arr_intersected_by_lin_subspace}]
Let $m\in \{n-d+1,\ldots, n\}$ (otherwise all terms in~\eqref{eq:lemma:faces_arr_intersected_by_lin_subspace} vanish).
Observe that each $m$-face $F$ from the fan $\cF(\cA)$ is contained in a unique $m$-dimensional linear subspace $M\in\cL(\cA)$ that can be represented as an intersection of hyperplanes from $\cA$. The $m$-dimensional cones $F\in\cF_{m}(\cA)$ with $F\subset M$ are the closures of the $m$-dimensional regions generated by the induced  arrangement $\cA|M=\{H\cap M:H\in\cA,M\nsubseteq H\}$ in $M$ and therefore we obtain
\begin{align*}
\#\{F\in\cF_m(\cA): F\cap L_d \neq\{0\}\}
&	=\sum_{D\in\cF_{m}(\cA)}\1_{\{L_d \cap \, D\neq\{0\}\}}\\
&	=\sum_{M\in\cL_{m}(\cA)}\:\sum_{D\in\cF_{m}(\cA):D\subset M}\ind_{\{(L_d\cap M)\cap\,  D\neq\{0\}\}}\\
&   =\sum_{M\in\cL_{m}(\cA)}\:\sum_{R\in\cR(\cA|M)}\ind_{\{(L_d\cap M)\cap\, R\neq\{0\}\}}.
\end{align*}
By our general position assumption, the subspace $L_d\cap M$ is of codimension $n-d$ in $M$ and additionally in general position with respect to $\cA|M$ in $M$. Let the coefficients of the characteristic polynomial of the arrangement $\cA|M$ in the ambient space $M$ be given by
\begin{align*}
\chi_{A|M}(t)=\sum_{k=0}^{m}(-1)^{m-k}a_k^Mt^k.
\end{align*}
Then, we can apply Lemma~\ref{lemma:number_regions} to the ambient linear subspace $M$ and arrive at
\begin{align}
\#\{F\in\cF_m(\cA): F\cap L_d \neq\{0\}\}
&=
2\sum_{M\in\cL_{m}(\cA)}(a_{n-d+1}^M + a_{n-d+3}^M + \ldots)\notag\\
&=
2(a_{n-d+1, m}+a_{n-d+3, m}+\ldots). \notag 
\end{align}
The evaluation of $\#\{F\in\cF_m(\cA): \relint F\cap L_d\neq\varnothing\}$ is similar. Without the general position assumption, the inequalities follow from~\cite[Lemma~3.5]{KVZ15}.
\end{proof}


\subsection{Face numbers of projected belt polytopes: Proof of Theorem~\ref{theorem:face_numbers_general_polytope}}\label{section:proof_face_numbers}
For the proof of the following Farkas' Lemma, we refer to~\cite[Lemma 2.4]{Amelunxen2017} and~\cite[Lemma~2.1]{HugSchneider2016}.
\begin{lemma}[Farkas]\label{lemma:Farkas}
Let $C\subset\R^n$ be a full-dimensional  cone and $L\subset\R^n$ a linear subspace. Then,
\begin{align*}
\inte(C)\cap L\neq\varnothing  \quad \Leftrightarrow \quad  C^\circ\cap L^\perp  = \{0\}.
\end{align*}
\end{lemma}
\begin{proof}[Proof of Theorem~\ref{theorem:face_numbers_general_polytope}]
Consider first the case when $P$ is full-dimensional.
Let $F$ be a $j$-face of $P$ and $0\le j<d\le n$ be given. Then, by~\cite{AS92} or~\cite[Proposition~5.3]{GKZ20_Random_Polytopes}, $GF$ is a $j$-face of $GP$ if and only if
\begin{align*}
\inte T_F(P)\cap\ker G= \varnothing
\end{align*}
since the general position assumption~\ref{label_GPG1} is satisfied. By Farkas' Lemma~\ref{lemma:Farkas}, this is equivalent to
\begin{align*}
(\ker G)^\perp\cap N_F(P)\neq \{0\}.
\end{align*}
Thus, using that $\cN(P)=\cF(\cA)$ and, in particular, $\{N_F(P):F\in\cF_j(P)\}=\cF_{n-j}(\cA)$, we obtain
\begin{align*}
f_j(GP)
	=\sum_{F\in\cF_j(P)}\ind_{\{GF\in\cF_j(GP)\}} 
	=\sum_{F\in\cF_j(P)}\ind_{\{(\ker G)^\perp\cap\, N_F(P)\neq\{0\}\}} 
	=\sum_{D\in\cF_{n-j}(\cA)}\1_{\{(\ker G)^\perp\cap \, D\neq\{0\}\}}.
\end{align*}
By~\ref{label_GPG2}, $(\ker G)^\perp$ is in general position w.r.t.\ $\cA$. By Lemma~\ref{lemma:number_faces_arr_intersected_by_lin_subspace} and its proof,
\begin{align}\label{eq:fulldimensional_case}
f_j(GP)
=
2\sum_{M\in\cL_{n-j}(\cA)}(a_{n-d+1}^M+a_{n-d+3}^M+\ldots)
=
2(a_{n-d+1, n-j}+a_{n-d+3, n-j}+\ldots),
\end{align}
which completes the proof in the full-dimensional case.

Now, suppose $p:=\dim P<n$. We want to restrict all arguments to the $p$-dimensional linear subspace $L$ satisfying $\aff P=t+L$ for some $t\in\R^n$, and then apply the already known full-dimensional case in the ambient space $L$. At first, we observe that $\rank (G|_L)=d$, since $\dim\ker (G|_L)=\dim(L\cap \ker G)=p-d\geq 0$ because $\ker G$ is in general position with respect to $P$ due to general position assumption~\ref{label_GPG1}. Furthermore, we need to verify whether the conditions~\ref{label_GPG1} and~\ref{label_GPG2} also hold in the restricted case where $n$ is replaced by $p$, $G$ is replaced by the restriction $G|_L$ of $G$ to $L$,  and $\cA$ is replaced by $\cA|L=\{H\cap L:H\in \cA,L\nsubseteq H\}=\{H\cap L:H\in \cA\}$. The last equation is due to $L^\perp\subset \lin N_F(P)$ for all faces $F$ of $P$, and therefore, $L^\perp\subset H$ for all $H\in\cA$, since the linear hull $\lin N_F(P)$ coincides with an intersection of hyperplanes from $\cA$. Thus, we also observe that the elements of $\cA|L$ and $\cA$ are in one-to-one correspondence via the mapping $H'\mapsto H'+L^\perp$ and, the inverse map is given by $H\cap L\mapsfrom H$.

Also, by~\ref{label_GPG1} for $P$ in $\R^n$, $\ker (G|_L)$ is in general position with respect to $K$, for each linear subspace $K$ such that $\aff F=t+K$ for some face $F$ of $P$, since
\begin{align*}
\dim (K\cap \ker (G|_L))=\dim(K\cap L\cap \ker G)=\dim(K\cap \ker G).
\end{align*}
Thus, \ref{label_GPG1} is also satisfied if we restrict all objects to $L$. Then,~\ref{label_GPG2} is also satisfied in the restricted version due to the equivalence of~\ref{label_GPG1} and~\ref{label_GPG2} proved in Theorem~\ref{theorem:gen_pos_general}. Thus, we can apply~\eqref{eq:fulldimensional_case} in the restricted case to obtain
\begin{align*}
f_j(G P)=2\sum_{M'\in\cL_{p-j}(\cA|L)}\big(a^{M'}_{p-d+1}+a^{M'}_{p-d+3}+\ldots\big),
\end{align*}
since $(\cA|L)|M'=\cA|M'$ and therefore $\chi_{(\cA|L)|M'}(t)=\chi_{\cA|M'}(t)$. Next we observe that the linear subspaces $M'\in\cL_{p-j}(\cA|L)$ are in one-to-one correspondence to the linear subspaces $M\in \cL_{n-j}(\cA)$ via $M'\mapsto M'+L^\perp =: M$. By the Whitney formula for the characteristic polynomial~\eqref{eq:def_char_polynomial_whitney} and the identity $\cA|(M'+L^\perp)=(\cA|M')+L^\perp$, we obtain the relation
\begin{align*}
\chi_{\cA|(M'+L^\perp)}(t) = \chi_{(\cA|M')+L^\perp}(t) = t^{n-p}\chi_{\cA|M'}(t),
\end{align*}
for all $M'\in\cL_{p-j}(\cA|L)$, and thus, $a^{M'}_{k}=a^{M'+L^\bot}_{k+n-p}$. Hence, we arrive at
\begin{align*}
f_j(G P)
	=2\sum_{M'\in\cL_{p-j}(\cA|L)}\big(a^{M'+L^\perp}_{n-d+1}+a^{M'+L^\perp}_{n-d+3}+\ldots\big)
	=2\sum_{M\in\cL_{n-j}(\cA)}\big(a^M_{n-d+1}+a^M_{n-d+3}+\ldots\big),
\end{align*}
which completes the proof.
\end{proof}

\subsection{Angle sums of belt polytopes: Proof of Theorem~\ref{theorem:angle_sums_belt_polytope}} \label{sec:proof_angle_sums_belt_polys}
We shall use the following generalized version of the Klivans-Swartz formula~\cite{KS11} due to Amelunxen and Lotz~\cite[Theorem~6.1]{Amelunxen2017}.
Consider a linear hyperplane arrangement $\cA$ in $\R^n$ whose $m$-th level characteristic polynomial is written in the form~\eqref{eq:char_poly_hyperplane_arr_m_level_coefficients}. Then, for all $0\leq k \leq m \leq n$ we have
\begin{equation}\label{eq:klivans_swartz_generalized}
\sum_{C\in \cF_{m}(\cA)} \upsilon_{k} (C) = a_{k,m}.
\end{equation}
To prove the first identity of Theorem~\ref{sec:proof_angle_sums_belt_polys}, we use the fact that $T_F(P)$ is the dual of $N_F(P)$ (which has dimension $n-j$), the definition of belt polytopes and then~\eqref{eq:klivans_swartz_generalized}:
\begin{equation}\label{eq:proof_sum_angle_belt_polys_upsilon}
\sum_{F\in \cF_j(P)} \upsilon_d(T_F(P))
=
\sum_{F\in \cF_j(P)}\upsilon_{n-d} (N_F(P))
=
\sum_{C\in \cF_{n-j}(\cA)} \upsilon_{n-d} (C)
=
a_{n-d,n-j}.
\end{equation}
To prove the second identity of Theorem~\ref{sec:proof_angle_sums_belt_polys}, we first suppose that $j<\dim P$. Then, for every $F\in \cF_j(P)$, the tangent cone $T_F(P)$ is not a linear subspace, and we can apply the conic Crofton formula~\eqref{eq:conic_crofton} together with~\eqref{eq:proof_sum_angle_belt_polys_upsilon}:
\begin{align*}
\sum_{F\in \cF_j(P)} \gamma_d(T_F(P))
&=
2 \sum_{F\in \cF_j(P)} \left( \upsilon_{d+1}(T_F(P)) + \upsilon_{d+3}(T_F(P)) + \ldots \right)\\
&=
2 (a_{n-d-1,n-j} + a_{n-d-3,n-j} + \ldots).
\end{align*}
In the remaining case $j=d= \dim P$, both sides of the above formula vanish. Indeed, the only $j$-dimensional face $F$ is $P$ itself,  $T_F(P)$ is a $d$-dimensional linear space and $\gamma_d(T_F(P)) = 0$. On the other hand, $\dim P = d$ implies that every linear subspace from $\cL(\cA)$ contains the $(n-d)$-dimensional linear subspace $M = (\aff P)^\perp$. The restriction of $\cL(\cA)$ to $M$ is empty, implying that all $a_{k, n-j}$ vanish.

\section{Proofs of the results on permutohedra}\label{sec:proofs_permutoh}


\subsection{Faces of permutohedra}\label{section:proof_faces_permuto}
Before starting with the proofs, let us mention the following well-known fact.
\begin{lemma}
The points  $(x_{\sigma(1)},\dots,x_{\sigma(n)})$ are indeed vertices of $\cP_n^A$ for all $\sigma\in\text{Sym}(n)$. Similarly, the points  $(\eps_1x_{\sigma(1)},\dots,\eps_nx_{\sigma(n)})$ are vertices of $\cP_n^B$ for all $\eps\in\{\pm 1\}^n$, $\sigma\in\text{Sym}(n)$.
\end{lemma} 
\begin{proof}
Let us explain this in the $B$-case. It suffices to prove the claim for the point $x = (x_1,\dots,x_n)$ of $\cP_n^B$, where $x_1\geq \ldots \geq x_n\geq 0$. Suppose that there are points $y=(y_1,\dots,y_n)\in\cP_n^B$ and $z=(z_1,\dots,z_n)\in\cP_n^B$ such that $x= (y+z)/2$.
By Lemma~\ref{lemma:rado_B}, we have $|y_1|\leq x_1$ and $|z_1|\leq x_1$. Thus, we have $y_1=z_1=x_1$. Given this, we can consider the second coordinate in the same way. Inductively, we obtain $y_i=z_i = x_i$ for all $i=1,\dots,n$, which means that $(x_1,\dots,x_n)$ is indeed a vertex of $\cP_n^B$.
\end{proof}



\subsubsection*{Faces of the permutohedra}
We now state an explicit characterization of the faces of both types of permutohedra.
Let $\cR_{n,j}$ be the set of all ordered partitions $\cB = (B_1,\dots,B_j)$ of the set $\{1,\dots,n\}$ into $j$ non-empty, disjoint and distinguishable subsets $B_1,\dots,B_j$. Furthermore, let $\cT_{n,j}$ be the set of all pairs $(\cB,\eta)$, where $\cB=(B_1,\dots,B_{j+1})$ is an ordered partition of the set $\{1,\dots,n\}$ into $j+1$ disjoint distinguishable subsets such that $B_1,\dots,B_j$ are non-empty, whereas $B_{j+1}$ may be empty or not, and $\eta:B_1\cup\ldots\cup B_j\to\{\pm 1\}$. In what follows, we write $\eta_i:=\eta(i)$ for ease of notation.

The elements of $\cR_{n,j}$ and $\cT_{n,j}$ are in bijective correspondence with the $j$-faces of the fans $\cF(\cA(A_{n-1}))$ and $\cF(\cA(B_{n}))$. The $j$-face in the fan of $\cA(A_{n-1})$ corresponding to a partition $\cB\in \cR_{n,j}$ is the set of all $x\in \R^n$ such that $x_{i_1} = x_{i_2}$ for all indices $i_1,i_2$ from the same $B_\ell$ and $x_{i_1} \leq x_{i_2}$ for all $i_1 \in B_{\ell_1}$ and $i_2 \in B_{\ell_2}$ with $1\leq \ell_1 < \ell_2 \leq j$. Furthermore, the $j$-face in the fan of $\cA(B_n)$ corresponding to the pair $(B,\eta)$ is the set of all $x\in \R^n$ such that $\eta_{i_1} x_{i_1} = \eta_{i_2} x_{i_2}$ for all indices $i_1,i_2$ from the same $B_\ell$ with $1\leq \ell \leq j$,   $\eta_{i_1} x_{i_1} \leq \eta_{i_2} x_{i_2}$ for all $i_1 \in B_{\ell_1}$ and $i_2 \in B_{\ell_2}$ with $1\leq \ell_1 < \ell_2 \leq j$, and $x_i=0$ for all $i\in B_{j+1}$.
It is easy to check that the cardinalities of $\cR_{n,j}$ and $\cT_{n,j}$ are given by $j!\stirlingsec{n}{j}$ and $\sum_{m=j}^n 2^m \binom{n}{m}j!\stirlingsec{m}{j} = 2^j j! \stirlingsecb{n}{j}$, respectively.

\begin{proposition}\label{prop:faces_permuto_A}
Suppose that $x_1>\ldots>x_n$.
Then, for $j\in\{0,\dots,n-1\}$, the $j$-dimensional faces of $\cP_n^A(x_1,\dots,x_n)$ are in one-to-one correspondence with the ordered partitions $\cB\in\cR_{n,n-j}$. The $j$-face corresponding to the ordered partition $\cB=(B_1,\dots,B_{n-j})\in\cR_{n,n-j}$ is given by
\begin{align*}
F_\cB
&=\conv\{(x_{\sigma(1)},\dots,x_{\sigma(n)}):\sigma\in I_\cB\}
\\
&=\bigg\{(t_1,\dots,t_n)\in\cP_n^A(x_1,\dots,x_n):\sum_{i\in B_1\cup\ldots\cup B_l}t_i=x_1+\ldots+x_{|B_1\cup\ldots\cup B_l|}\;\forall l=1,\dots,{n-j-1}\bigg\}.
\end{align*}
Here, $I_\cB\subset\text{Sym}(n)$ is the set of all permutations $\sigma\in\text{Sym}(n)$ such that
\begin{multline*}
\sigma(B_1) = \{1,\ldots,|B_1|\},
\;\;
\sigma(B_2) = \{|B_1|+1,\ldots,|B_1\cup B_2|\},
\;\;
\ldots,\\
\sigma(B_{n-j}) = \{|B_1\cup\ldots \cup B_{n-j-1}|+1,\ldots,n\}.
\end{multline*}
\end{proposition}

\begin{proposition}\label{prop:faces_permuto_B}
Suppose that $x_1>\ldots>x_n>0$.
Then, for $j\in\{0,\dots,n\}$, the $j$-dimensional faces of $\cP_n^B(x_1,\dots,x_n)$ are in one-to-one correspondence with the pairs $(\cB,\eta)\in\cT_{n,n-j}$.
The $j$-face corresponding to the pair $(\cB,\eta)$, where $\cB=(B_1,\dots,B_{n-j+1})$, is given by
\begin{align}
F_{\cB,\eta}
&	=\conv\{(\eps_1x_{\sigma(1)},\dots,\eps_nx_{\sigma(n)}):(\sigma,\eps)\in I_{\cB,\eta}\}\label{eq:faces_permutoB1}\\
&	=\bigg\{t\in\cP_n^B(x_1,\dots,x_n):\sum_{i\in B_1\cup\ldots\cup B_l}\eta_i t_i=x_1+\ldots+x_{|B_1\cup\ldots\cup B_l|}\;\forall l=1,\dots,{n-j}\bigg\}\label{eq:faces_permutoB2}.
\end{align}
Here, $I_{\cB,\eta}\subset\textup{Sym}(n)\times\{\pm 1\}^n$ is the set of all pairs $(\sigma,\eps)\in\textup{Sym}(n)\times\{\pm 1\}^n$ such that
\begin{multline*}
\sigma(B_1) = \{1,\ldots,|B_1|\},
\;\;
\sigma(B_2) = \{|B_1|+1,\ldots,|B_1\cup B_2|\},
\;\;
\ldots,\\
\sigma(B_{n-j+1}) = \{|B_1\cup\ldots \cup B_{n-j}|+1,\ldots,n\}
\end{multline*}
and $\eps_i=\eta_{i}$ for all $i\in B_1\cup\ldots\cup B_{n-j}$, while the remaining $\eps_i$'s take arbitrary values in $\{\pm 1\}$.
\end{proposition}

Proofs of Proposition~\ref{prop:faces_permuto_A} can be found in~\cite[pp.~254-256]{barvinok_book} or in~\cite[Section~5.3]{yemelichev_book}. Without proof, versions of the same proposition are stated in~\cite[Proposition 2.6]{Postnikov2009} and in Exercise 2.9 on p.~96 of~\cite{Bjoerner1999}.  For completeness, we will provide a proof of Proposition~\ref{prop:faces_permuto_B} (which may also be known).

\begin{proof}[Proof of Proposition~\ref{prop:faces_permuto_B}]
Let $F\in\cF(\cP_n^B)$ be a face of $\cP_n^B(x_1,\ldots,x_n)$ with $x_1>\ldots>x_n>0$. Either, we have $F=\cP_n^B$, which means there is nothing to prove, or there is a supporting hyperplane $H=\{t\in\R^n:\alpha_1t_1+\ldots+\alpha_nt_n=b\}$ for some $\alpha=(\alpha_1,\dots,\alpha_n)\in\R^n\backslash \{0\}$ and $b\in\R$ such that
\begin{align}\label{eq:supp_hyperplane_B}
H\cap \cP_n^B=F\quad\text{and}\quad \cP_n^B\subset H^-:=\{t\in\R^n:\alpha_1t_1+\ldots+\alpha_nt_n\le b\}.
\end{align}
Without loss of generality, we may assume that $\alpha_1\ge \ldots\ge \alpha_n\ge 0$ (otherwise apply a signed permutation of $\{1,\dots,n\}$ to $(\alpha_1,\ldots,\alpha_n)$ and all other objects). Then,
\begin{align}\label{eq:alphas_B}
\underbrace{\alpha_1=\dots=\alpha_{i_1}}_{\text{group $1$}}>\underbrace{\alpha_{i_1+1}=\ldots=\alpha_{i_2}}_{\text{group $2$}}>\ldots>\underbrace{\alpha_{i_{m-1}+1}=\ldots=\alpha_{i_{m}}}_{\text{group $m$}}>\underbrace{\alpha_{i_{m}+1}=\ldots=\alpha_n=0}_{\text{group $m+1$}},
\end{align}
for some $m\in\{1,\dots,n\}$ and $1\le i_1<\ldots<i_{m}\le n$.  Note that for $i_{m}=n$, no $\alpha_i$'s are required to be zero, which means that the last group is empty. Then, $\cP_n^B\subset H^-$ implies that
\begin{align*}
\alpha_1\eps_1x_{\sigma(1)}+\ldots+\alpha_n\eps_nx_{\sigma(n)}\le b, \quad \text{ for all } \eps=(\eps_1,\dots,\eps_n)\in\{\pm 1\}^n,\sigma\in\text{Sym}(n).
\end{align*}
The first equation of~\eqref{eq:supp_hyperplane_B} implies that there is a pair $(\sigma', \eps')\in \text{Sym}(n)\times \{\pm 1\}^n$ such that
\begin{align*}
\alpha_1\eps_1' x_{\sigma'(1)}+\ldots+\alpha_{i_{m}}\eps_{i_{m}}'x_{\sigma'(i_{m})}
=
\alpha_1\eps_1'x_{\sigma'(1)}+\ldots+\alpha_n\eps_n'x_{\sigma'(n)}
=
b.
\end{align*}
Since the $\alpha_i$'s and the $x_i$'s are non-increasing and non-negative, the swapping lemma (see, e.g., \cite[p.~254]{barvinok_book}) states that $\alpha_1\eps_1x_{\sigma(1)}+\ldots+\alpha_n\eps_nx_{\sigma(n)}$ attains its maximal value if  we choose  $\eps_i=+1$ and $\sigma(i) =i$  for all $i\in \{1,\ldots,n\}$. It follows that, in fact, we have
\begin{align}\label{eq:alpha_x_eq_b}
\alpha_1x_1+\ldots+\alpha_{n}x_{n}=b.
\end{align}
Denote the groups of indices appearing in~\eqref{eq:alphas_B} by
\begin{equation} \label{eq:B_1_ldots_B_Fall}
B_1=\{1,\dots,i_1\},\;\;
\dots,\;\;
B_{m}=\{i_{m-1}+1,\dots,i_{m}\},
\;\;
B_{m+1}=\{i_{m}+1,\dots,n\},
\end{equation}
where $B_{m+1}$ may or may not be empty. Defining in our case $\eta_i :=1$ for all $i\in \{1,\ldots,i_m\}$  we obtain that, under~\eqref{eq:B_1_ldots_B_Fall}, the set $I_{\cB,\eta}$ consists of all pairs $(\sigma,\eps)\in\text{Sym}(n)\times\{\pm 1\}^n$ such that
$$
\sigma(B_1)=B_1,\;\; \dots, \;\; \sigma(B_m)=B_m, \;\; \sigma(B_{m+1})=B_{m+1}
$$
and $\eps_i = \eta_i$ for all $i\in \{1,\ldots,i_m\}$. Consequently, from~\eqref{eq:alpha_x_eq_b} and~\eqref{eq:alphas_B} it follows that
\begin{align}\label{eq:alpha_x_eq_b_perm}
\alpha_1\eps_1x_{\sigma(1)}+\ldots+\alpha_{n}\eps_{n}x_{\sigma(n)}=b
\quad  \text{ for all } (\sigma,\eps)\in I_{\cB,\eta}.
\end{align}
Furthermore, it follows from \eqref{eq:alpha_x_eq_b}, \eqref{eq:alphas_B} and the swapping lemma that
\begin{align}\label{eq:alpha_x_less_b_strictly}
\alpha_1\eps_1x_{\sigma(1)}+\ldots+\alpha_{n}\eps_{n}x_{\sigma(n)}<b
\quad  \text{ for all }  (\sigma,\eps)\in(\text{Sym}(n)\times\{\pm 1\}^n)\backslash I_{\cB,\eta}.
\end{align}
Indeed, if $(\sigma,\eps)\notin I_{\cB,\eta}$, then there is the possibility that we have a strictly negative term on the left-hand side of~\eqref{eq:alpha_x_less_b_strictly} which means that we could make it strictly larger by changing the sign of this term. Thus, we can assume all these terms to be non-negative. Then, $(\sigma,\eps)\notin I_{\cB,\eta}$ implies that there is a pair of indices $1\le i<j\le n$ such that $\alpha_i>\alpha_j$ und $x_{\sigma(i)}<x_{\sigma(j)}$ and we can apply the swapping lemma to strictly increase the left-hand side.

According to~\eqref{eq:alpha_x_eq_b_perm} and~\eqref{eq:alpha_x_less_b_strictly}, the vertices $(\eps_1 x_{\sigma(1)},\dots,\eps_nx_{\sigma(n)})$ with $(\sigma,\eps)\in I_{\cB,\eta}$ are the only vertices of $\cP_n^B$ that belong to the supporting hyperplane $H$.  It follows from~\cite[Proposition~2.3]{Ziegler_LecturesOnPolytopes} that $F$ is the convex hull of these vertices, that is $F=F_{\cB,\eta}$, where
\begin{align}\label{eq:faces_permutoB3}
F_{\cB,\eta}:=\conv\big\{(\eps_1x_{\sigma(1)},\dots,\eps_n x_{\sigma(n)}):(\sigma,\eps)\in I_{\cB,\eta}\big\}.
\end{align}

Essentially the same argument shows that, conversely, a set of the form $F_{\cB,\eta}$ is a face of $\cP_n^B$.  At the beginning, we applied a signed permutation to all objects including $(\alpha_1,\dots,\alpha_n)$ to achieve that the $\alpha_i$'s are non-increasing and non-negative. Applying the inverse signed permutation proves that the faces of $\cP_n^B$ coincide with the sets of the form $F_{\cB,\eta}$ as defined in~\eqref{eq:faces_permutoB3}, for some pair $(\cB,\eta)\in\cT_{n,m}$.
Furthermore, for two different pairs $(\cB',\eta'),(\cB'',\eta'')$ we have $I_{\cB',\eta'}\neq I_{\cB'',\eta''}$, which implies that the corresponding  sets $F_{\cB',\eta'}$ and $F_{\cB'',\eta''}$ are different, since their sets of vertices are different.
Finally, the polytope $F_{\cB,\eta}$, for $(\cB,\eta)\in\cT_{n,m}$, is isometric to the direct product $\cP_{|B_1|}^A \times\ldots \times \cP_{|B_m|}^A\times \cP_{|B_{m+1}|}^B$, which follows from the description of the vertices of $F_{\cB,\eta}$. It follows that $\dim F_{\cB,\eta}=n-m$.

Now, we  prove the equivalence of the representations~\eqref{eq:faces_permutoB1} and~\eqref{eq:faces_permutoB2}.
To this end, we take some pair $(\cB,\eta)\in \cT_{n,m}$, assuming without restriction of generality that
\begin{align*}
B_1:=\{1,\dots,i_1\},\;\;
B_2:=\{i_1+1,\dots,i_2\},
\;\;
\dots,
\;\;
B_{m+1}:=\{i_{m}+1,\dots,n\},
\end{align*}
where $1\le i_1<\ldots<i_{m}\le n$ for some $m\in \{1,\dots,n\}$, and $\eta_i=1$ for $i\in \{1,\dots,i_m\}$. Our goal is to prove that $F_{\cB,\eta} = M$, where
\begin{align*}
M:= \bigg\{t\in\cP_n^B(x_1,\dots,x_n):\sum_{i\in B_1\cup\ldots\cup B_l}t_i=x_1+\ldots+x_{|B_1\cup\ldots\cup B_l|}\;\forall l=1,\dots,{m}\bigg\}.
\end{align*}
The inclusion  $F_{\cB,\eta}\subset M$ holds trivially and we only need to prove that $M\subset F_{\cB,\eta}$.

Let $\alpha=(\alpha_1,\ldots,\alpha_n)\in \R^n\backslash\{0\}$ be such that  condition~\eqref{eq:alphas_B} holds. The above arguments show that the hyperplane $H=\{t\in\R^n:\alpha_1t_1+\ldots+\alpha_nt_n=b\}$ with $b:=\alpha_1x_1+\ldots+\alpha_n x_n$ is a supporting hyperplane of the face $F_{\cB,\eta}$, that is
$$
H\cap \cP_n^B = F_{\cB,\eta}\quad\text{and}\quad \cP_n^B\subset H^-:=\{t\in\R^n:\alpha_1t_1+\ldots+\alpha_nt_n\le b\}.
$$
Suppose now that there is some $y\notin F_{\cB,\eta}$ such that $y\in M\subset\cP_n^B\subset H^-$. This already yields
\begin{align*}
\alpha_1y_1+\ldots+\alpha_{i_m}y_{i_m}=\alpha_1y_1+\ldots+\alpha_ny_n<t=\alpha_1x_1+\ldots+\alpha_{i_m}x_{i_m},
\end{align*}
since $y\in H^-$, but $y\notin H$. It follows that
\begin{align*}
&\alpha_{i_{m}}(y_1+\ldots+y_{i_m})+\sum_{l=1}^{m-1}(\alpha_{i_{l}}-\alpha_{i_{l+1}})\sum_{i=1}^{i_l}y_i\\
&	\quad=\alpha_1y_1+\ldots+\alpha_{i_m}y_{i_m}\\
&	\quad<\alpha_1x_{1}+\ldots+\alpha_{i_m}x_{i_m}\\
&	\quad=\alpha_{i_{m}}(x_{1}+\ldots+x_{i_m})+\sum_{l=1}^{m-1}(\alpha_{i_{l}}-\alpha_{i_{l+1}})\sum_{i=1}^{i_l}x_i,
\end{align*}
which is a contradiction to $y\in M$. This proves that both representations~\eqref{eq:faces_permutoB1} and~\eqref{eq:faces_permutoB2} are equivalent.
\end{proof}

\subsection{Normal fans of permutohedra: Proofs of Theorems~\ref{theorem:normalfan_permuto_A} and~\ref{theorem:normalfan_permuto_B}} \label{subsec:proof_normal_fans_permutoh}
First we need to prove a simple lemma concerning the interior of a polytope.

\begin{lemma}\label{lemma:relintP}
Let $P\subset\R^n$ be a polytope
defined by the affine inequalities
\begin{align*}
P=\{x\in\R^n:l_1(x)\le 0,\dots,l_m(x)\le 0\}
\end{align*}
for some $m\in\N$ and affine-linear functions $l_i(x)=\langle x,y_i\rangle+b_i$, where $y_i\in\R^n\backslash\{0\}$ and $b_i\in\R$, $i=1,\dots,m$. Then, we have
\begin{align*}
\inte P=\{x\in\R^n:l_1(x)<0,\dots,l_m(x)<0\}.
\end{align*}
\end{lemma}

\begin{proof}
Suppose $x\in\R^n$ satisfies the conditions $l_1(x)<0,\dots,l_m(x)<0$. Since the functions $l_1,\dots,l_m$ are continuous, we also have $l_1(y)<0,\dots,l_m(y)<0$ for all $y$ in some small enough neighborhood of $x$. Thus, $x$ belongs to $\inte P$.

Now let  $x\in P$ satisfy $l_i(x)=0$ for some $i\in\{1,\dots,m\}$. Then, in each neighborhood of $x$,  we can find a point $y$   with $l_i(y)>0$. This means that $x\notin  \inte P$, thus completing the proof.
\end{proof}

\begin{proof}[Proof of Theorem~\ref{theorem:normalfan_permuto_A}]
Let $x_1>\ldots>x_n$ be given. Our aim is to prove that $\cN(\cP_n^A(x_1,\dots,x_n))=\cF(\cA(A_{n-1}))$.
From Proposition~\ref{prop:faces_permuto_A} we know that each $j$-face of $\cP_n^A$, for a $j\in\{0,\dots,n-1\}$, is uniquely defined by an ordered partition $\cB=(B_1,\dots,B_{n-j})\in\cR_{n,n-j}$ of the set $\{1,\dots,n\}$ and given by
\begin{align*}
F_{\cB}=\bigg\{(t_1,\dots,t_n)\in\cP_n^A:\sum_{i\in B_1\cup\ldots\cup B_l}t_i=x_1+\ldots+x_{|B_1\cup\ldots\cup B_l|}
\text{ for all } l=1,\dots,{n-j}\bigg\}.
\end{align*}
Now, take a point $t\in\relint F_{\cB}$. We claim that $t$ satisfies the following conditions:
\begin{align}\label{eq:cond_relint_A1}
\sum_{i\in B_1\cup\ldots\cup B_l}t_i=x_1+\ldots+x_{|B_1\cup\ldots\cup B_l|}\quad\forall l=1,\dots,n-j,
\end{align}
and
\begin{align}\label{eq:cond_relint_A2}
\sum_{i\in M}t_i<x_1+\ldots+x_{|M|}\quad\forall M\subset \{1,\dots,n\}:M\notin\{B_1,B_1\cup B_2,\dots,B_1\cup\ldots\cup B_{n-j}\}.
\end{align}
To prove this, consider the affine subspace
\begin{align*}
L_\cB:=\bigg\{(t_1,\dots,t_n)\in\R^n:\sum_{i\in B_1\cup\ldots\cup B_l}t_i=x_1+\ldots+x_{|B_1\cup\ldots\cup B_l|}
\text{ for all } l=1,\dots,{n-j}\bigg\},
\end{align*}
which is of dimension $j$ since the conditions are linearly independent. Then, by Lemma~\ref{lemma:rado_A}, we can represent $F_\cB$ as  the set of points $(t_1,\dots,t_n)\in L_\cB$ such that
\begin{align*}
\sum_{i\in M}t_i\le x_1+\ldots+x_{|M|}\quad \forall M\subset \{1,\dots,n\}:M\notin\{B_1,B_1\cup B_2,\dots,B_1\cup\ldots\cup B_{n-j}\}.
\end{align*}
Since $\dim F_\cB=j$, the characterization of $\relint F_\cB$ in~\eqref{eq:cond_relint_A2} follows from Lemma~\ref{lemma:relintP} applied to the ambient affine subspace $L_\cB$ instead of $\R^n$.

Now, we want to determine the tangent cone $T_{F_\cB}(\cP_n^A)$. By definition, the tangent cone is given by
\begin{align*}
T_{F_{\cB}}(\cP_n^A)=\{v\in\R^n:t+\eps v\in\cP_n^A \text{ for some $\eps>0$}\},
\end{align*}
where $t\in\relint F_{\cB}$.
By Lemma~\ref{lemma:rado_A}, for a $v\in\R^n$, the condition $t+\eps v\in\cP_n^A$ holds for some $\eps>0$ if and only if
\begin{align*}
\sum_{i=1}^n (t_i+\eps v_i) = x_1+\ldots+x_n\quad\text{and}\quad \sum_{i\in M} (t_i+\eps v_i) \le x_1+\ldots+x_{|M|}\;\,\forall M\subset \{1,\dots,n\}.
\end{align*}
 Since $t_1+\ldots+t_n=x_1+\ldots+x_n$, the first condition is satisfied if and only if $v_1+\ldots+v_n=0$.
We observe that if we choose $\eps>0$ small enough, the second condition is satisfied for all sets $M\subset\{1,\dots,n\}$ such that $M\notin\{B_1,B_1\cup B_2,\dots,B_1\cup\ldots\cup B_{n-j}\}$, due to~\eqref{eq:cond_relint_A2}. For the sets $B_1,B_1\cup B_2,\dots,B_1\cup\ldots\cup B_{n-j}$, we obtain that
\begin{align*}
\sum_{i\in B_1\cup\ldots\cup B_l}v_i\le 0,
\end{align*}
by~\eqref{eq:cond_relint_A1}. Therefore, the tangent cone is given by
\begin{align*}
T_{F_{\cB}}(\cP_n^A)=\bigg\{v\in\R^n:v_1+\ldots+v_n=0,\;\sum_{i\in B_1\cup\ldots\cup B_l}v_i\le 0\;\forall l=1,\dots,n-j-1\bigg\}.
\end{align*}
Thus, the corresponding normal cone is given by
\begin{align*}
N_{F_\cB}(\cP_n^A)=T_{F_{\cB}}(\cP_n^A)^\circ=\{x\in\R^n:\forall 1\le l_1\le l_2\le n-j\;\forall i_1\in B_{l_1},i_2\in B_{l_2}, \text{ we have } x_{i_1}\ge x_{i_2}\}.
\end{align*}
Note that the conditions of $N_{F_\cB}(\cP_n^A)$ imply $x_{i_1}=x_{i_2}$ for all $i_1,i_2\in B_l$, $l=1,\dots,n-j$. The cone $N_{F_\cB}(\cP_n^A)$ is an $(n-j)$-dimensional cone in the fan $\cF(\cA(A_{n-1}))$ and it is easy to check that, going through all ordered partitions $\cB\in\cR_{n,n-j}$, we obtain all $(n-j)$-dimensional cones of the fan $\cN(\cA(A_{n-1}))$; see, e.g., \cite[Section 2.7]{KVZ2019_multi_analogue}. This completes the proof.
\end{proof}

\begin{proof}[Proof of Theorem~\ref{theorem:normalfan_permuto_B}]
Fix some $x_1>\ldots>x_n>0$. Our aim is to prove that $\cN(\cP_n^B(x_1,\dots,x_n))=\cF(\cA(B_n))$.
From Proposition~\ref{prop:faces_permuto_B} we know that each $j$-face of $\cP_n^B$, for a $j\in\{0,\dots,n\}$, is uniquely defined by a pair $(\cB,\eta)\in\cT_{n,n-j}$, where $\cB=(B_1,\dots,B_{n-j+1})$, and given by
\begin{align*}
F_{\cB,\eta}=\bigg\{(t_1,\dots,t_n)\in\cP_n^B:\sum_{i\in B_1\cup\ldots\cup B_l}\eta_it_i=x_1+\ldots+x_{|B_1\cup\ldots\cup B_l|}\;\forall l=1,\dots,{n-j}\bigg\}.
\end{align*}
Now, we claim that
\begin{align}\label{eq:claim_tangent_cone_B}
T_{F_{\cB,\eta}}(\cP_n^B)=\bigg\{v\in\R^n:\sum_{i\in B_1\cup\ldots\cup B_l}\eta_iv_i\le 0\;\forall l=1,\dots,n-j\bigg\}.
\end{align}
To prove this, take a point $t\in\relint F_{\cB,\eta}$. Then,
\begin{align}\label{eq:cond_relint_B1}
\sum_{i\in B_1\cup\ldots\cup B_l}\eta_it_i=x_1+\ldots+x_{|B_1\cup\ldots\cup B_l|}\quad\forall l=1,\dots,n-j
\end{align}
and
\begin{align}\label{eq:cond_relint_B2}
\sum_{i\in M}|t_i|<x_1+\ldots+x_{|M|}\quad\forall M\subset \{1,\dots,n\}:M\notin\{B_1,B_1\cup B_2,\dots,B_1\cup\ldots\cup B_{n-j}\}.
\end{align}
This can  easily be justified in the same way as in the $A$-case using Lemmas~\ref{lemma:rado_B} and~\ref{lemma:relintP}. Note that~\eqref{eq:cond_relint_B1} implies that $\sgn t_i=\eta_i$ for all $i\in B_1\cup\ldots\cup B_{n-j}$ such that $t_i\neq 0$. Otherwise, if $\eta_{i_0}=-\sgn t_{i_0}$ for some $i_0\in\{1,\dots,n\}$ with $t_{i_0}\neq 0$, we would have
\begin{align*}
\sum_{B_1\cup\ldots\cup B_l}|t_i|>\sum_{B_1\cup\ldots\cup B_l}\eta_it_i=x_1+\ldots+x_{|B_1\cup\ldots\cup B_l|}
\end{align*}
for some $l\in\{1,\dots,n-j\}$ in contradiction to $t\in \cP_n^B$.

Now, recall that the tangent cone is defined by
\begin{align*}
T_{F_{\cB,\eta}}(\cP_n^B)=\{v\in\R^n:t+\eps v\in\cP_n^B \text{ for some $\eps>0$}\}
\end{align*}
for $t\in \relint F_{\cB,\eta}$. In view of the characterization of points in $\cP_n^B$ stated in Lemma~\ref{lemma:rado_B}, it follows  that $v\in T_{F_{\cB,\eta}}(\cP_n^B)$ if and only if there exists an $\eps>0$ such that
\begin{align*}
\sum_{i\in M}|t_i+\eps v_i|\le x_1+\ldots+x_{|M|}\quad\forall M\subset\{1,\dots,n\}.
\end{align*}
For all $M\subset\{1,\dots,n\}$ with $M\notin\{B_1,B_1\cup B_2,\dots,B_1\cup\ldots\cup B_{n-j}\}$ this condition is satisfied due to~\eqref{eq:cond_relint_B2} provided $\eps>0$ is small enough.
If $t_i\neq 0$ for all $i\in B_1\cup\ldots \cup B_{n-j}$, the remaining conditions are equivalent to
\begin{align}\label{eq:condB2}
\sum_{i\in B_1\cup \ldots \cup B_l} \eta_i(t_i+\eps v_i)\le x_1+\ldots+x_{|B_1\cup\ldots\cup B_l|}
\quad
\forall l=1,\ldots,n-j.
\end{align}
 This follows from the fact that $\sgn(t_i+\eps v_i)=\sgn t_i=\eta_i$ for $\eps>0$ chosen small enough. By~\eqref{eq:cond_relint_B1}, we obtain
\begin{align*}
\sum_{i\in B_1\cup\ldots\cup B_l}\eta_iv_i\le 0,
\end{align*}
for all $l=1,\dots,n-j$.  This proves~\eqref{eq:claim_tangent_cone_B}. At this point, it remains to prove that $t_i\neq 0$ for all $i\in B_1\cup\ldots\cup B_{n-j}$. To this end, assume that $t_i=0$ for some $i\in B_l$ and some $l\in\{1,\dots,n-j\}$. Defining $D_i:=(B_1\cup\ldots\cup B_l)\backslash \{i\}$, we have
\begin{align*}
\sum_{j\in D_i}\eta_jt_j=\sum_{j\in B_1\cup\ldots\cup B_l}\eta_jt_j=x_1+\ldots+x_{|B_1\cup\ldots\cup B_l|},
\end{align*}
due to~\eqref{eq:cond_relint_B1}. If $D_i=B_1\cup\ldots\cup B_m$ for some $m<l$, we obtain
\begin{align*}
x_1+\ldots+x_{|B_1\cup\ldots\cup B_m|}=\sum_{j\in D_i}\eta_jt_j=x_1+\ldots+x_{|B_1\cup\ldots\cup B_l|},
\end{align*}
in contradiction to $x_i>0$ for all $i=1,\dots,n$. If, on the other hand, $D_i\neq B_1\cup\ldots\cup B_m$ for all $m<l$, we have
\begin{align*}
x_1+\ldots+x_{|B_1\cup\ldots\cup B_l|}=\sum_{j\in D_i}\eta_jt_j<x_1+\ldots+x_{|D_i|},
\end{align*}
by~\eqref{eq:cond_relint_B2}. This is a contradiction to $D_i\subset B_1\cup\ldots\cup B_l$ proving that $t_i\neq 0$ for all $i\in B_1\cup\ldots\cup B_{n-j}$.

Thus, the normal cone of $\cP_n^B$ at $F_{\cB,\eta}$ is given by
\begin{align*}
N_{F_{\cB,\eta}}(\cP_n^B)
&	=T_{F_{\cB,\eta}}(\cP_n^B)^\circ\\
&	=\big\{x\in\R^n:\forall 1\le l_1\le l_2\le n-j\;\forall i_1\in B_{l_1},i_2\in B_{l_2}, \text{ we have } \eta_{i_1}x_{i_1}\ge \eta_{i_2}x_{i_2}\geq 0;\\
&\hspace{2,23cm}\forall i\in B_{n-j+1}\text{ we have }x_i=0\big\}.
\end{align*}
The cone $N_{F_\cB}(\cP_n^B)$ is an $(n-j)$-face of a Weyl chamber of type $B_n$ and we can observe that, going through all pairs $(\cB,\eta)\in\cT_{n,n-j}$, we obtain all $(n-j)$-dimensional cones of the fan $\cN(\cA(B_n))$; see, e.g.,~\cite[Section~2.4]{KVZ2019_multi_analogue}. This completes the proof.
\end{proof}


\subsection{Permutohedra and zonotopes: Proof of Theorem~\ref{theo:permuto_vs_zonotopes}}\label{subsec:proofs_permotohedra_as_zonotopes}
We will prove both the $A$- and the $B$-case together and assume that $x_1>\ldots>x_n$ and $x_1>\ldots>x_n>0$, respectively. In the book of Ziegler~\cite[Example~7.15]{Ziegler_LecturesOnPolytopes}, it is shown  that $\cP_n^A(n,n-1,\dots,1)$ is a zonotope. By shifting and rescaling, we obtain that
\begin{align*}
\cP_n^A\big(a+(n-1)b,a+(n-2)b,\dots,a+b,a\big)
\end{align*}
is also a zonotope for each $a\in\R$ and $b>0$. Similarly, we can also prove that $\cP_n^B(n,n-1,\dots,1)$ is a zonotope and therefore also $\cP_n^B(a+(n-1)b,a+(n-2)b,\dots,a+b,a)$ for each $a>0$ and $b>0$. This follows from the representation of $\cP_n^B(n,n-1,\dots,1)$ as the following Minkowski sum of line segments:
\begin{align*}
\cP_n^B(n,n-1,\dots,1)
&	=\sum_{1\le i<j\le n}\bigg[-\frac{e_i-e_j}{2},\frac{e_i-e_j}{2}\bigg]+\sum_{1\le i<j\le n}\bigg[-\frac{e_i+e_j}{2},\frac{e_i+e_j}{2}\bigg]+ \sum_{1\le i\le n}[-e_i,e_i].
\end{align*}
To prove this, we observe that this Minkowski sum is invariant under signed permutations of the coordinates. Additionally, we can compute the vertices of this Minkowski sum, that is, the points of the Minkowski sum that maximize a linear function $v\mapsto \langle c,v\rangle$, $\R^n\to\R$,  for a vector $c\in\R^n$, provided the maximizer is unique. Applying a signed permutation, we may assume that $c_1\geq c_{2} \geq \ldots \geq c_n\geq 0$. On the line segment $[-\frac{e_i-e_j}{2},\frac{e_i-e_j}{2}]$, the function $v\mapsto \langle c,v\rangle$ is uniquely maximized by the right-hand boundary $\frac{e_i-e_j}{2}$ provided $c_i > c_j$. For $c_i=c_j$, the maximizer is not unique. Therefore, we may assume that $c_1>c_{2} > \ldots > c_n> 0$.   Then, the unique maximizer of $v\mapsto \langle c,v\rangle$ is given by the sum of the right-hand boundaries of the line segments:
\begin{align*}
v
&	=\sum_{1\le i<j\le n}\frac{e_i-e_j}{2}+\sum_{1\le i<j\le n}\frac{e_i+e_j}{2}+\sum_{1\le i\le n}e_i\\
&	=\sum_{1\le i<j\le n}e_i+\sum_{1\le i\le n}e_i\\
&	=(n,n-1,\dots,1)^\top.
\end{align*}
Hence, the vertices of the above Minkowski sum have the form $(\eps_1\sigma(n),\eps_2\sigma(n-1),\dots,\eps_n\sigma(1))$, for all $\eps\in\{\pm 1\}^n, \sigma\in\text{Sym}(n)$.
This proves the representation of $\cP_n^B(a+(n-1)b,a+(n-2)b,\dots,a+b,a)$ as the Minkowski sum of certain line segments. In particular, this polytope  is a zonotope.

To prove the other direction, assume that $\cP_n^A(x_1,\dots,x_n)$ with $n\geq 3$ is a zonotope and use that a polytope $P$ is a zonotope if and only if every $2$-dimensional face of $P$ is centrally symmetric~\cite[p.~200]{Ziegler_LecturesOnPolytopes}. From Proposition~\ref{prop:faces_permuto_A} we know that the convex hull $F$ of the six points
\begin{align*}
(x_{\sigma(1)},x_{\sigma(2)},x_{\sigma(3)},x_4,x_5,\dots,x_n),\quad\sigma\in\text{Sym}(3),
\end{align*}
is a $2$-face of $\cP_n^A(x_1,\dots,x_n)$. This face is centrally symmetric around some $a=(a_1,\dots,a_n)$. This means that for each vertex $z$ of $F$, also $2a-z$ is a vertex of $F$. Thus, we obtain the conditions
\begin{align*}
2a_1-x_1,2a_1-x_2,2a_1-x_3\in\{x_1,x_2,x_3\}.
\end{align*}
From $x_1>x_2>x_3$, we obtain
\begin{align*}
2a_1-x_1=x_3\quad\text{and}\quad 2a_1-x_2=x_2
\end{align*}
and therefore also $x_1+x_3=2x_2$. This yields $x_3-x_2=x_2-x_1$. Analogously, by considering more general $2$-faces of $\cP_n^A(x_1,\dots,x_n)$, one proves that $x_{j+1}-x_j=x_j-x_{j-1}$ for all admissible $j$. Thus, $x_1,\dots,x_n$ are in arithmetic progression.

The proof that $x_1,\dots,x_n$ form an  arithmetic progression if $\cP_n^B(x_1,\dots,x_n)$ is a zonotope follows in the same way as in the $A$-case since the $2$-faces of $\cP_n^A$ considered above are also $2$-faces of $\cP_n^B$ according to Proposition~\ref{prop:faces_permuto_B}.
\hfill $\Box$

\vspace*{0.5cm}
\bibliography{bib}
\bibliographystyle{abbrv}

\subsection*{Statements and Declarations}
Supported by the German Research Foundation under Germany’s Excellence Strategy EXC 2044 – 390685587, Mathematics M\"unster: Dynamics - Geometry - Structure and by the DFG priority program SPP 2265 Random Geometric Systems. The authors have no relevant financial or non-financial interests to disclose. We are grateful to the unknown referees for useful suggestions and, in particular, for bringing the concept of belt polytopes to our attention.

\subsection*{Data availability}
Data sharing is not applicable to this article as no datasets were generated or analysed during the current study.



\end{document}